\documentclass[11pt,reqno]{article}
\usepackage[usenames,dvipsnames]{xcolor}
\usepackage{multirow} 
\usepackage{mathrsfs,amssymb,amsfonts,amsthm,amsmath,amsbsy}
\usepackage{dsfont,verbatim,color,fancyhdr}
\usepackage{graphicx,float} 
\usepackage[hang]{caption} 
\usepackage{enumerate} 
\usepackage[normalem]{ulem}
\usepackage[bookmarks,bookmarksnumbered,colorlinks=true,pdfstartview=FitV, linkcolor=blue,citecolor=blue,urlcolor=blue]{hyperref}
\usepackage{pdfsync}

\usepackage{color, url}
\usepackage{graphicx}
\definecolor{comcolor}{rgb}{0.9,0.3,0.3}
\definecolor{starcolor}{rgb}{0.3,0.3,0.9}
\definecolor{hscolor}{rgb}{0.9,0.6,0.5}
\definecolor{darkgreen}{rgb}{0.1,0.6,0.3}

\newtheorem{thm}{Theorem}[section]
\newtheorem{lemma}[thm]{Lemma}
\newtheorem{corollary}[thm]{Corollary}
\newtheorem{prop}[thm]{Proposition}

\theoremstyle{definition}
\newtheorem{defn}[thm]{Definition}
\newtheorem{example}[thm]{Example}

\newtheorem{rem}[thm]{Remark}

\newcommand{\be}[1]{\begin{equation}\label{#1}}
\newcommand{\ee}{\end{equation}}
\newcommand{\ba}{\begin{array}}
\newcommand{\ea}{\end{array}}
\newcommand{\bal}{\begin{aligned}}
\newcommand{\eal}{\end{aligned}}

\newcommand{\R}{\mathbb{R}}

\newcommand{\N}{\mathbb{N}}

\newcommand{\Z}{\mathbb{Z}}
\newcommand{\E}{\mathbb{E}}

\renewcommand{\P}{\mathbb{P}}

\newcommand{\mbb}[1]{\mathbb{#1}}
\newcommand{\mb}[1]{\mathbf{#1}}
\newcommand{\wt}[1]{\widetilde{#1}}

\newcommand{\Ex}[1]{\ensuremath{\mathbb{E} \left[#1 \right]}}
\newcommand{\Prob}[1]{\ensuremath{\mathbb{P} \left(#1 \right)}}

\title{\vspace{-15mm}\bf  
Duality and fixation in $\Xi$-Wright-Fisher processes with frequency-dependent selection
}
\author{
\textsc{Adri\'an Gonz\'alez Casanova
\footnote{Supported by the German Research Foundation through the Priority Programme 1590 \textit{Probabilistic Structures in Evolution.}}
}\\
\emph{Weierstrass Institute Berlin}\\[1mm]
\textsc{Dario Span\`o}\\
\emph{University of Warwick}
}
\date{\today}

\fancypagestyle{plain}{
\fancyhead{}
}

\begin{document}

\maketitle
\thispagestyle{empty}

\begin{abstract}
A two-types, discrete-time population model with finite, constant size is constructed, allowing for a general form of frequency-dependent selection and skewed offspring distribution. Selection is defined based on the idea that individuals first choose a (random) number of \emph{potential} parents from the previous generation and then, from the selected pool, they inherit the type of the fittest parent. The probability distribution function of the number of potential parents per individual thus parametrises entirely the selection mechanism. Using duality, weak convergence is then proved both for the allele frequency process of the selectively weak type and for the population's ancestral process. The scaling limits are, respectively, a two-types {$\Xi$-Fleming-Viot} jump-diffusion process with frequency-dependent selection, and a branching-coalescing process with general branching and simultaneous multiple collisions. Duality also leads to a characterisation of the probability of extinction of the selectively weak allele, in terms of the ancestral process' ergodic properties. 
\medskip

\noindent \textit{Key words and phrases}: Cannings models, frequency-dependent selection, moment duality, ancestral processes, branching-coalescing stochastic processes, fixation probability, $\Xi$-Fleming-Viot processes.\medskip

\noindent\textit{AMS 2010 subject classifications}: 60G99, 60K35, 92D10, 92D11, 92D25.
\end{abstract}

\section{Introduction}
Modelling selection is acknowledged to be one of the most delicate problems in mathematical population genetics. A variety of hypotheses have been proposed to describe how competing allelic types jostle against each other in trying to propagate successfully their type in the next generation (\cite{Ki68, KO72, GIL84, WE04}). Despite the complexity of the debate on the concept of selection itself, there is general agreement on the idea that an appropriate measure of strength of fitness of a given allelic type is the probability of its eventual fixation in the population, conditional on a given starting frequency. \\
Fixation is the event that the frequency of the allelic type will eventually reach 1 and stay there forever.
For the exact calculation of the probability of fixation, the notion of duality has recently proved to be a formidable tool by means of which to combine efficiently information coming from both a forward-in-time (allele frequencies diffusion) and a backward-in-time (ancestral process) analysis of the population considered. 
In particular, it has been established (e.g. \cite{G14, F13, EGT10}) that, in case of weak selection, the fate of a selectively disadvantaged allele is intrinsically connected, via duality, to the long-term dynamics of the so-called \emph{block-counting process} describing, at each time in the past, how many non-mutant lineages are alive in the ancestral graph representing the population's genealogy.
Unless the population is neutral, such graphs (the so-called \emph{Ancestral Selection Graph}, ASG \cite{NK97, KN97, N99})
are not trees, but coalescing/branching processes where single lineages, when traced backwards in time, may split into one \emph{true} and one (or two, in some forms of balancing selection, see \cite{N99,WS09}) \emph{virtual} parental lines, at rates which encode the difference in fitness between the two allelic types. 

{The ASG induces} a notion of selection associated to the idea of lineages splitting into multiple potential parents. This observation is a crucial starting point for the purposes of this paper.\\

In this paper we will construct a class of population models incorporating two key features: on one hand, they allow for a general form of frequency-dependent selection, leading to a genealogy with multiple branching of lineages; on the other hand, they allow for the possibility of high-fecundity extreme reproductive events ($\Xi$-events), leading to a genealogy with simultaneous multiple coalescence of lineages. 
This class can be interpreted as a class of Cannings models with selection, as defined in \cite{LL07}. For models in this class, to the best of our knowledge, a unified treatment of the scaling-limit allele frequency dynamics, the corresponding ancestry, as well as the probability of fixation, has not been carried out in full generality yet. 
\\
We will begin with a construction of a two-types, discrete-time population model with constant size $N$ and non-overlapping generations. Our model of selection, whose construction is laid out in Section 2,  is based on the idea that individuals first choose a (random) number of \emph{potential parents} from the previous generation and then, from the selected pool, they inherit the type of the fittest parent. The probability distribution function of the number of potential parents per individual thus parametrises  entirely the selection mechanism. In particular, the model is non-neutral whenever the individuals are allowed to choose more than one potential parent with positive probability. Thus rather than encoding limit properties of selection, in this paper multiple parents become part of the definition of selection itself. \\
As for extreme reproductive events, these are modelled by assuming that, in some generations chosen at random, distinct individuals make correlated choices of their potential parents, the correlation being driven by a background measure $\Xi$ on the infinite simplex. This is in line with many modern constructions of coalescent processes with simultaneous multiple collisions and, in fact, our construction may be viewed as a discrete-time analogue of the Poisson-construction of $\Xi$-Moran models presented in \cite{BBetal09}, plus selection. On the other hand, by keeping track of each individual's potential ancestors backward in time, our model yields as a finite-population, discrete-time analogue of an ASG, with non-overlapping generations, converging to a general ASG scaling limit with multiple branching  (as opposed to binary-only or ternary-only branching) as well as simultaneous multiple coalescences of lineages.\\


We can show that in the limit as $N\to\infty$, with time appropriately rescaled,  the process of the allele frequency of the selectively weaker type converges in distribution to a two-type $\Xi$-Fleming-Viot jump-diffusion process in $[0,1]$ with frequency-dependent selection, solution to the stochastic differential equation (SDE) \begin{eqnarray}\label{sdeintro}
d X_t&=&- \kappa s(X_t)X_t(1-X_t)dt+\sqrt{ \sigma X_t(1-X_t)}dB_t\notag\\&&+\int_{(0,1]}\int_{(0,1]}\sum_{i=1}^\infty y_i\Big(\mathbb{I}_{\{u_i\le X_{t-} \}}-X_{t-}\Big) \widetilde{N}(d t,d y,d \overline u),
\end{eqnarray}
where $\kappa$ and $\sigma$ are non-negative constants and $s(x)$ is a power series function with positive, non-increasing coefficients: $\kappa$ measures the strength of the selective pressure, $s$ describes its shape as a function of the allele frequencies and $\sigma$ represents the effective population size (in population genetics terminology, the strength of the random genetic drift). 
By \lq\lq time appropriately rescaled\rq\rq we mean that time is measured in units of $\rho_N^{-1},$ where $\rho_N$ (assumed to be $o(1/N)$) is the probability, for any individual, of choosing more than one potential parent. Further assumptions on the mean number of potential parents will also be needed (see condition \emph{(iv)} of Proposition \ref{lemma1}). 

The first addend on the right-hand side of \eqref{sdeintro} is the one accounting for selection. The frequency-dependent selection function $s$ will turn out (see later, Proposition \ref{lemma1}) to be entirely determined by the limit distribution of the number of potential parents per individual. Indeed, denote by $\varphi(x)$ the probability generation function (pgf) of the distribution of such a number, conditional on it being larger than one. It will appear (see Remark \ref{r:gen}, equation \eqref{bpdrift}) that $$s(x)=\frac{1-\varphi(x)}{1-x}.$$ \\
The second term in \eqref{sdeintro} is the classical Wright--Fisher diffusion and the third term accounts for the jumps induced by extreme reproductive events, driven by a Poisson process $\wt N$ whose intensity measure depends on $\Xi$ and regulates the frequencies' jumps and sizes when the population is infinite. Note that our convergence result 
applies to any choice of $\Xi$-measure. \\
To the best of our knowledge, the existence itself of a solution to the augmented SDE \eqref{sdeintro}, encompassing both frequency-dependent selection and $\Xi$-extreme reproduction dynamics, has not been proved before. A proof will be given in Lemma \ref{l:exist} based on a recent work of {Gonz\'alez Casanova et al}  \cite{GPP16}. We also believe that, although frequency-dependent selection in Cannings models with non-overlapping generations have been considered in the literature (\cite{LL07}), scaling limit approximations have so far been derived only under the assumption that, under neutrality, the population is in domain of attraction of Kingman's coalescent i.e. its limit behaviour is purely diffusive. Our convergence result of Proposition \ref{lemma1} overcomes such a limitation.\\
Several special cases of \eqref{sdeintro} are nevertheless well known:
\begin{itemize}
\item For $\kappa s(x)=0$, the SDE reduces to the two-type version of the so-called neutral $\Xi$-Fleming-Viot SDE (see \cite{BBetal09}, Section 5.3), whose genealogy is described by a coalescent tree process with simultanous multiple collisions, or $\Xi$-coalescent (\cite{P99, S00, MS01}).\\
 \item Without the jump component $(\wt N\equiv 0)$, one recognises in \eqref{sdeintro} the SDE of a general frequency-dependent selection diffusion process in $[0,1]$ which captures many of the most popular models of selection in genetics, including the following notable examples:
\begin{itemize}
\item $s(x)= \bar s$ for a constant $\bar s$. This is the classical model of \emph{weak selection} (\cite{Ki68, KO72}). 
\item $s(x)=\lambda-\nu x$ where $\lambda>-\nu>0$.  This is a case of \emph{balancing selection.} 
The parameters $\lambda$ and $\nu$ have been described in terms of stochastic evolutionary games (see e.g. \cite{LL07, PV12} and references therein).  In our framework, they will be essentially tail probabilities of the distribution of the potential parents' number. The scaling limit genealogy is encoded by an ASG with ternary branches (\cite{N99}). 
\end{itemize}
More examples of Cannings models with frequency-dependent selection leading to purely diffusive scaling limits can be found in \cite{LL07}, where the selection mechanism is defined via evolutionary game theory.
\item  With non-zero jump component $\wt N$, the particular case of \eqref{sdeintro} where $s(x)\equiv 1$ and the parameter measure $\Xi$ is concentrated on $[0,1]$ ($\Lambda$-coalescent) has been studied independently by Griffiths \cite{G14} and Foucart \cite{F13}, although a derivation of the SDE from a pre-limiting model, or a discussion on the existence of its solution, was not within their aims. 
\end{itemize}
The cited work of Griffiths and Foucart (but see also \cite{EGT10}) forms the main background for the methodology and results on duality and fixation proposed in this paper. Both authors apply a duality method to describe the probability of fixation of the advantageous allele. In particular, Griffiths' methodology relies on the following clever interpretation of the generator of the two-type Lambda Fleming-Viot under neutrality:
\begin{equation}
 Af(x)=\frac{1}{2}\Ex{x(1-x)f{''}(x(1-W)+VW)},\label{bobgenintro}
 \end{equation}
where $V$ is a Uniform random variable in $[0,1]$, $W=SY$ with $Y$ a $\Lambda/\Lambda([0,1])$-distributed random variable and $S$ a size-biased uniform random variable in $[0,1]$, and all the variables on the right-hand-side are independent. See \cite{G14}, equation (7).
Our own derivation is crucially based on an extension of Griffiths' approach to more general $\Xi$-coalescent dynamics. 
  See Lemma \ref{l:xibob} of this paper.\\
  We will show (in Proposition \ref{pr:limdual}) that moment duality holds between the process $(X_t)$ solution to \eqref{sdeintro}, 
  and the block-counting process $(D_t)$ associated to a branching-coalescing random graph (the limit ASG), with generator given by 
\begin{eqnarray}
Lf(n)=\kappa \sum_{i=0}^\infty \pi_i[f(n+i-1)-f(n)] + \sigma{n\choose 2}[f(n-1)-f(n)]+\overline{L}f(n)
\label{L_n_intro}
\end{eqnarray}
for every $n\in\N$ and $f:\N\rightarrow \R$ in $C_2,$ where {$\pi_i$ is a sequence of constants depending on the distribution of the number of potential parents and}  $\bar{L}$ denotes the generator of the $\Xi$-coalescent (see equation \eqref{lbar} for details). Without the $\bar L$ component, $L$ is the generator of a branching process with logistic growth (\cite{AL05}).\\
   Moment duality means namely that, for every $x\in[0,1],n\in\N,$
 \begin{equation}
 \Ex{X_t^n\mid X_0=x}=\Ex{x^{D_t}\mid D_0=n}.\label{momdualintro}
 \end{equation}
\\
  We will exploit moment duality to study the probability of fixation for the process $(X_t)$ of the selectively weaker allele frequency: we will prove that, for any given initial frequency $x$, the process almost surely gets extinct (i.e. $X_t=0$ eventually) if and only if $(D_t)$ does not have a stationary distribution. The weaker allele may survive and in fact even reach fixation ($X_t=1$) if $(D_t)$ has a stationary distribution. The probability of fixation of the fitter allele will depend on the stationary distribution of $(D_t)$. This is the content of Lemma \ref{lemma:selection21}.   Necessary and sufficient conditions for either scenario to hold will be found in Theorem \ref{THM} to depend on a critical value $\kappa^*$ for the parameter $\kappa$ measuring the \emph{total selection pressure} in \eqref{sdeintro}. The critical value will depend on the mean number $\beta$, say, of potential parents per individuals in a branching event as well as on the coalescent parameter measure $\Xi$: indeed, the process $(D_t)$ will reach stationarity if and only if $\kappa\geq\kappa^*$ where 
\begin{equation}
\kappa^*:=\frac{1}{2\beta}\Ex{\frac{1}{\sum_{i=1}^\infty {Z^*_i}^2}\frac{1}{W(1-W)}},
\label{criticalk}
\end{equation}
so long as $\beta$ and $\kappa^*$ are finite and $\Xi$ satisfies some mild admissibility conditions (see Definition \ref{def:admiss}). In \eqref{criticalk}, $Z^*_i=Z_i/|Z|,$ $W=S|Z|$, $|Z|:=\sum_iZ_i$ for $Z=(Z_1,Z_2,\ldots)$ a random sequence with distribution $\Xi/\Xi(\nabla_\infty)$, where $\nabla_\infty:=\{x_1\geq x_2\geq\cdots 0:\sum_i x_i\leq 1\}$ is the ranked infinite simplex, $S$ is a size-biased uniform random variable, $(B_i:i\in\N)$ are \emph{i.i.d.} Bernoulli random variables with parameter $x$ and $Z,S,V, (B_i)$ are independent. \\
The problem of describing the long-term behaviour of branching-coalescing processes including multiple collisions has been addressed in \cite{GPP16} but only under the assumption that $\sigma>0$, i.e. that Kingman-like coalescence events occur with positive rates. The method in \cite{GPP16}  crucially relies on stochastic domination and does not allow to analyse if and how, in absence of a Kingman component, the competing events of branching and coalescence balance each other out or, on the contrary, one of them eventually prevails. Our Theorem \ref{THM} aims to fill this missing gap, thus in Section \ref{sec:fixation} we will assume $\sigma=0$ which, interestingly, turns out to be the only case with finite $\beta$ where a threshold $\kappa^*$ can be found.

\subsection{Outline of the paper}
The paper is structured as follows: in Section \ref{sec:diWF} our two-type discrete model will be introduced in terms of a random graph with vertex set in a portion of $\Z^2$ (generation $\times$ label of individual), with random edges denoting potential ancestral relation between individuals in successive generations. This approach will allow us to define both the frequency process and the ancestral process on the same probability space. 
In the same Section, a form of duality relationship, the so-called \emph{sampling duality} (see \cite{M99}), between the allele frequency process and the block-counting process of the population's ancestral graph, will be proved.  We will also derive the one-step transition probabilities of the frequency process of the weak allelic type. In both representations, a central role will be played by the pgf of the distribution of the number of potential parents per individual, per generation.\\
In Section \ref{sec:SDE} we will prove convergence, with time appropriately rescaled, of the weak allele's frequency process to the process $(X_t)$ solving the SDE \eqref{sdeintro}. In particular, we will show existence of a strong solution to \eqref{sdeintro}. We will thus establish moment duality \eqref{momdualintro} with a branching-coalescing process, i.e. the Markov chain with generator \eqref{L_n_intro}.\\
In Section \ref{sec:fixation} we will first prove our $\Xi$-version of Griffiths' representation \eqref{bobgenintro} for the generator associated to the frequency process' SDE \eqref{sdeintro} and then we will use moment duality to determine our criterion for fixation based on the critical value $\kappa^*$ for the selection pressure parameter $\kappa$.

\section{Discrete models with selection}\label{sec:diWF}
The goal of this section is to define a two-types, discrete-time population model, with finite constant size $N$ and non-overlapping generations, with frequency-dependent selection. We start with a first formulation of the model without $\Xi$-reproductive events.
\subsection{Selection without extreme reproductive events}\label{sec:noext}
We denote the allelic type space with $\{0,1\}$ and adopt the notation $\N=\{1,2,\ldots\}$. Type 1 will denote the selectively advantageous allele. We parametrise the strength of selection via a probability distribution $Q_N$ on $\N\cup\{\infty\}$. The reproduction mechanism works as follows.\\
 \
\\
\emph{(1). Choice of potential parents.}
At each generation $g\in\Z$, every individual $i$ $(i=1,\ldots,N)$ chooses, independently, a random number $K_{(g,i)}$ of \emph{potential parents} among the $N$ individuals in the previous generation, $g-1$, where $K_{(g,i)}$ is a random variable with distribution $Q_N$. The choice is with replacement in the sense that, given $K_{(g,i)}=k$, the individual $i$ chooses its $k$ parents by sampling $k$ labels independently and uniformly at random from $\{1,\ldots,N\}$. However, the ancestry of the model will retain information only about all the \emph{distinct} potential parents chosen by each individual, i.e. parents chosen more than once will be included in the ancestry only once.  \\ 
\\
\emph{(2). Choice of type.} 
Types 1 or 0 are arbitrarily assigned to all the individuals at a starting generation $g=0$, say. For every subsequent generation, each individual takes on the allelic type 0 if and only if all its potential parents carry the type 0. 
If at least one of the parents is of type 1, then the inherited allele will be 1.
\ \\
\ \\
\emph{(3). Actual vs virtual parents.} Although the distinction between virtual and actual parents will not play a significant role in the derivation of our main result, in order to stress the connection with the ASG of \cite{KN97} and \cite{N99}, we can stipulate that the \emph{actual parent} of $i$ is the individual with the lowest label among all the potential parents carrying the same type as $i$. All other potential parents will be considered as \emph{virtual parents}.
 \subsubsection{Wright-Fisher random graph}
 We shall now provide a random graph representation of the model, with the benefit of embedding in the same probability space both the forward in time allele frequency process and the ancestry of the population. Consider the set of vertices
$$V_N:=\Z\times\{1,\ldots,N\}
$$ For every $v=(g,i)\in V_N$ we denote with $g(v)=g$ the generation of the individual $v$ and with $i(v)=i$ its label. From every $v\in V_N$, an edge is drawn from $(g(v)-1,l)$ to $v$ if $l$ is one of the potential parents of $v$. The random set $E_N$ of all such edges thus depends on the following random variables:
\begin{itemize}
\item[(i)] The collection $K=(K_{v}:v\in V_N)$ of \emph{i.i.d.} $Q_N$ random variables, indicating how many potential parents are chosen by each individual at each generation. 
\item[(ii)] The collection $$L^0(K)=\{L^0_{v}\}_{v\in V_N}=\{L^{0}_{(v,1)},L^{0}_{(v,2)},\ldots, L^{0}_{(v,K_v)}\}_{v\in V_N}$$ of \emph{i.i.d.} random variables uniformly distributed over $\{1,\ldots,N\}$, where, for every $v=(g,i)$, $L^0_v$ lists the labels of all the potential parents selected by the individual $i$ at generation $g$.
\end{itemize}
For every $v$ let $J_v\leq K_v$ be the number of distinct parents appearing in $L^0_{v}=(L^{0}_{(v,1)},L^{0}_{(v,2)},\ldots, L^{0}_{(v,K_v)})$, and with $\wt L^0_v=( \wt L^{0}_{(v,1)},\dots,\wt L^{0}_{(v,J_v)})$ denote their labels. 

\begin{defn}\label{WF-fss} For every $N\in\N$ and $Q_N\in{\cal P}(\N\cup\{\infty\})$, the \emph{Wright-Fisher graph with (frequency-dependent) selection $Q_N$} is the graph with vertex set $V_N$ and the edge set \begin{equation}E^0_N:=\big \{\{(g(v)-1,\wt L^{0}_{(v,j)}),v\}: j=1,\ldots,J_v, v\in V_N\big\}.\label{E_N}
\end{equation}
\end{defn}
\subsubsection{Allele frequencies}
 Let us assign to each vertex of a fixed starting generation $g=0$ either type $0$ or $1$ arbitrarily.  Each vertex $v$ in each of the subsequent generations will be of type $1$ if and only if it is connected in $(V_N,E_N)$ to at least one vertex of type $1$ in the restriction $(V_N,E_N)\cap\left\{v\in V_N: g(v)\geq 0\right\}$. We are interested in the evolution of the type-$0$ frequencies. Let $\xi(v)$ denote the type of vertex $v$ and with $$X^N_g=1-\frac{1}{N}\sum_{\{v:g(v)=g\}}\xi(v)$$ the frequency of type-$0$ individuals at generation $g=0,1,\ldots$. \\
 Denote with $\varphi_{Q_N}$ the probability generating function of $Q_N$. Consider the sampling function $$\mu_N(x)=\Prob{\xi(v)=0\mid\, X^N_{g(v)-1}=x}$$
that is, the probability that a vertex $v$ is of type $0$ if the frequency of type $0$ individuals in generation $g(v)-1$ is $x$. This event occurs if and only if all the potential parents of $v$ are of type $0$. Since $(K_v)$ is an \emph{i.i.d.} sequence, this probability does not depend on $v$ and since the values of the coordinates in $L^0_v=(L^0_{(v,1)},\ldots,L^0_{(v,K_v)})$ do not depend on $K_v$, 
\begin{align}\label{samplingm}
\mu_{N}(x)=&\sum_{k=1}^\infty \Prob{\xi(g(v)-1,L^{0}_{(v,j)})= 0\text{ for all }j\leq K_v, K_{v}= k\mid\, X^N_{g(v)-1}=x}\nonumber
\\=&\sum_{k=1}^\infty x^k\ Q_N( K_{v}= k)=\varphi_{Q_N}(x).
\end{align}
The following proposition is an obvious consequence of \eqref{samplingm}, and of the fact that individuals choose their potential parents independently.
 \begin{prop}\label{pr:tf1}
 In a Wright-Fisher graph with selection parameter $Q_N$, the type-0 allele frequency process $(X^N_g:g\in\N)$ evolves as a time-homogeneous Markov chain on the state space $[N]/N:=\{0,1/N,\ldots,(N-1)/N, 1\}$ with transition probabilities
 \begin{equation*}
 \Prob{X^N_g=m/N\mid\ X^N_{g-1}=x}={N\choose m}\varphi_{Q_N}(x)^m(1-\varphi_{Q_N}(x))^{N-m},\ \ m=0,1,\ldots,N,
 \end{equation*}
 for every $x\in [N]/N.$
 \end{prop}
 
\begin{example}[Weak selection]
With the choice 
\begin{equation}Q_N(K_{v}>m)= s_N^{m-1},\ \ \ m=1,2,\ldots\label{Q_N}
 \end{equation}
 for some $s_N>0$ (geometric distribution), the sampling probability $\mu_N$ becomes
 \begin{equation}\mu_{N}(x)=\frac{x(1-s_N)}{1-x+x(1-s_N)}.\label{classicsel}\end{equation}
 In other words, the allele frequency process $(X_g^N:g\in\N)$ of the Wright-Fisher graph $(V_N, E_N, Q_N)$ with geometric $Q_N$ coincides with the classical Wright-Fisher model with weak selection coefficient $s_N$ (\cite{Ki68, KO72, KT75}).
 \end{example}
 
\subsection{Selection with extreme reproductive events}
Now we will extend the model of Section \ref{sec:noext} in such a way to include the possibility of high fecundity events ($\Xi$-events), where the offspring of one or more individuals may replace a non-negligible (relative to the population size) proportion of the population in the next generation. The reproduction mechanism is somehow a discrete-time analogue to the Poisson construction of the $\Xi$-Moran model proposed in \cite{BBetal09}. To this purpose, we introduce a random background formed by a sequence of \emph{i.i.d.} Bernoulli trials $H=\{H_g:g\in\Z\}\in\{0,1\}^\infty$, with probability of success $\gamma_N\in[0,1]$ and a sequence $Z=\{Z_g:g\in \Z\}$ of \emph{i.i.d.} $\nabla_\infty$-valued random elements with common distribution $\Xi$. $H$ and $Z$ are assumed to be independent. For every $g\in\Z$, $H_g=1$ (respectively, $H_g=0$) indicates that at generation $g$ extreme reproduction does (respectively does not) occur. $Z$ will give the expected sizes of extreme reproductive events, when they occur. We assume that reproduction depends on $(H,Z)$ as follows:
\begin{itemize}
\item[(a)] every individual $v\in V_N$ samples, independently, a random number of potential parents $K_{v}$ from the previous generation, where $K_{v}$ has distribution $Q_N$; 
\item[(b)] given $K_{v}=k$ the individual chooses the labels of its $k$ potential parents by sampling $k$ \emph{i.i.d.} random variables with random distribution 
\begin{equation}\eta^*_g:=H_{g} \eta_{g}+(1-H_{g})U_N ,\label{etastar}
\end{equation}
where $g=g(v)$, $U_N$ is the uniform distribution on $\{1,\ldots,N\}$ and
\begin{equation}\eta_g:=\sum_{m=1}^\infty Z_{(g,m)}\delta_{Y^*_{(g,m)}}+(1-|Z_g|)U_N\label{eta}
\end{equation}
where $Y^*=\{Y^*_{(g,m)}:g\in\Z,m\in\N\}$ is a sequence of \emph{i.i.d.} random variables, independent of $(H,Z),$ each with uniform distribution on $\{1,\ldots, N\},$  $|Z_g|:=\sum_{m=1}^\infty Z_{(g,m)}$ and $H_g, Z_g, Y^*_g$ are independent.
\item[(c)] Each individual $v$ inherits type 1 if and only if at least one of its virtual parents is of type 1. 
\end{itemize}
In other words, at step (b), if $H_g=0$ (no extreme event occurs), all the potential parents of each individual are chosen independently, uniformly at random, exactly as described in Section \ref{sec:noext}. If $H_g=1$ (extreme event occurs), then the potential parents are sampled independently according to the random measure $\eta_g$. To explain how such a measure works, one might think of each potential parent being first assigned either to a group $m$ $(m=1,2,\ldots)$ with probability $Z_{(g,m)}$ or to a residual group $m=0$ with probability $1-|Z_g|;$ then, all the members of the same group $m=1,2,\ldots$ will choose collectively the same label uniformly at random, independently of all other groups, while each of the members in the residual group $m=0$ will make its own individual choice independently, uniformly at random. 

\subsubsection{$\Xi$-Wright-Fisher graph with selection} The random graph representing the population will be formed by the vertex set $V_N$ and a random edge set $E_N$ depending on the following random variables:
\begin{itemize}
\item[(i)] The random background $(H,Z)$ with distribution $\text{Ber}(\gamma_N)^{\otimes \infty}\otimes \Xi^{\otimes\infty}$;
\item[(ii)] The collection $K=(K_{v}:v\in V_N)$ of \emph{i.i.d.} $(Q_N)$ random variables, indicating how many potential parents are chosen by each individual at each generation;
\item[(iii)] The collection of potential parents $$L(K)=\{L_{v}\}_{v\in V_N}=\{L_{(v,1)},L_{(v,2)},\ldots, L_{(v,K_v)}\}_{v\in V_N}$$ of \emph{i.i.d.} random variables where, for each $v=(g,i)\in V_N$ and $j\in\N$, the distribution of $L_{(v,j)}$ is $\eta^*_g.$
\end{itemize}
Define $J_v\leq K_v$ the number of distinct parents sampled by $v$ and let $\{\wt L_{(v,1)},\ldots,\wt L_{(v,J_v)}\}$ be their labels.

\begin{defn}\label{WF-fss}
For every $N\in\N$, $\gamma_N\in[0,1],$ $Q_N\in{\cal P}(\N\cup\{\infty\})$ and $\Xi\in{\cal P}(\nabla_\infty)$, the \emph{$(\Xi, \gamma_N,Q_N)$-Wright-Fisher graph (with frequency-dependent selection $Q_N$)} is the graph with vertex set $V_N$ and the edge set 
\begin{equation}E_N=\big \{\{(g(v)-1,\wt L_{(v,j)}),v\}: j=1,\ldots,J_v, v\in V_N\big\}.\label{xi-E_N}
\end{equation}
\end{defn}

\subsubsection{$\Xi$-allele frequencies} 
A key property needed to prove both convergence and duality is the following Proposition. Denote 
\begin{equation}
\label{Y(x)}
Y(x)=\sum_{i=1}^\infty B_iZ_{i}+x(1-|Z|)
\end{equation}
for $Z=(Z_1,Z_2,\ldots)$ with distribution $\Xi$, and $(B_i:i=1,2,\ldots)$ \emph{i.i.d.} Bernoulli $(x)$, independent of $Z$.
\begin{prop}\label{pr:sfunct}
Let $(X^N_g:g\in\N)$ be the 0-type allele frequency process in a $(\Xi,\gamma_N, Q_N)$-Wright-Fisher graph. For every $g$, the probability that all the individuals in a sample of $n$ $(n\leq N)$ from generation $g$ will be of type 0, given  $X^N_{g-1}=x$, is
\begin{equation}
S(x,n)=(1-\gamma_N){\varphi_{Q_N}(x)}^n+\gamma_N \nu_N(x,n),
\label{sfunct}
\end{equation}
where $\varphi_{Q_N}(x)$ is the pgf of $Q_n$ and
\begin{equation}
\nu_N(x,n):=\Ex{\left(\varphi_{Q_N}\left(Y(x)\right)\right)^n}
\end{equation}
where $Y(x)$ is defined as in \eqref{Y(x)}.
\end{prop}

\begin{proof}
Given $X^N_{g-1}=x$, if $H_g=0$ the conditional probability that an individual $v$ from generation $g(v)=g$ is of type 0 is $\mu_N(x)$ as in \eqref{samplingm}. By independence,
$$\Prob{ \bigcap_{r=1}^n\{\xi(g, r)=0\}\,\mid\,H_{g}=0,X^N_{g-1}=x }={\mu_N(x)}^n={\varphi_{Q_N}(x)}^n.$$
 If $H_g=1$, the individual $v$ samples each of its \emph{i.i.d.} potential parents $(L_{v,1},\ldots)$ from $\eta_g$. If $Y^*$ is a random individual sampled uniformly from generation $(g-1)$, then, given $X^N_{g-1}=x$, the variable
$$B:=\mbb{I}\left[\xi(g-1,Y^*)=0\mid X^N_{g-1}=x\right]$$
is a Bernoulli random variable with probability of success $x$. From the form of $\eta_g$ then
\begin{align*}
&\,\P\left[\xi(g-1,L_{v,1})=0\mid X^N_{g-1}=x, H_g=1\right]\\
&\ \ \ \,=\Ex{\sum_{i=1}^\infty Z_{(g,i)}\Prob{\xi(g-1,Y^*_{g,i})=0\mid X^N_{g-1}=x}+x(1-|Z_g|)}\\
&\ \ \ \,=\Ex{\sum_{i=1}^\infty Z_{(g,i)}B_i+x(1-|Z_g|)}=\Ex{Y(x)}=x,
\end{align*}
where $(B_i)$ is a collection of \emph{i.i.d.} Bernoulli random variables with parameter $x$, independent of $Z_g$. Since $v$ chooses $K_{v}$ parents independently from $\eta_g$, then
\begin{align*}
&\,\P(\xi(v)=0\,\mid\,H_{g}=1,X^N_{g-1}=x)\\
&\,= \sum_{k=0}^\infty \Prob{\bigcap_{j=1}^{k}\xi(g-1, L_{(v,j)})=0\}\mid\,H_{g}=1,X^N_{g-1}=x}Q_N(K_v=k)\\
&\,=\Ex{\sum_{k=0}^\infty\left(\sum_{i=1}^\infty B_iZ_{(g,i)}+x(1-|Z_g|)\right)^kQ_N(K_v=k)}\\
&\,=\Ex{\varphi_{Q_N}\left(Y(x))\right)}=\nu_N(x,1).
\end{align*}
Similarly, when the choices of $n$ individuals are considered, then by independence of the random variables $\{K_{(g,i)}\}_{i=1,\ldots,N}$,
\begin{align*}
&\,\Prob{\bigcap_{r=1}^n\{\xi(g, r)=0\}\mid\,H_{g}=1,X^N_{g-1}=x}\\
&\,=\Ex{\left(\varphi_{Q_N}\left(\sum_{i=1}^mB_iZ_{(g,i)}+x(1-|Z_g|)\right)\right)^n}=\nu_N(x,n).
\end{align*}
Since the probability that $H_g=1$ is $\gamma_N,$ the result follows.
\end{proof}
By a similar argument it is easy to derive a description of the one-step 0-type allele frequency process.
\begin{corollary}\label{cor:tf}
Let $(X^N_g:g\in\N)$ be the 0-type allele frequency process in a $(\Xi,\gamma_N,Q_N)$-Wright-Fisher graph $(V_N, E_N)$. Then $X^N_g$ evolves as a Markov chain in $[N]/N$ with one-step transition probabilities
\begin{eqnarray}
&&\Prob{X^N_g=j/N\mid X^N_{g-1}=x}\notag\\
&&\ \ \ =(1-\gamma_N)\ {\rm Bin}(j;N,\varphi_{Q_N}(x))
+ \gamma_N\ \Ex{{\rm Bin}(j;N,\varphi_{Q_N}(Y(x)))}\label{xi_tf},\end{eqnarray}
for $ j\in\{1,\ldots,N\}, x\in[N]/N$,
where ${\rm Bin}(\cdot; n,p)$ is the binomial probability mass function with parameter $(n,p)$, $Y(x)$ is as in \eqref{Y(x)} and $\varphi_{Q_N}$ is the pgf of $Q_N$.
\end{corollary}

\subsection{Ancestry and duality}
\label{sec:a&d}
Now we will introduce the ancestral process induced by the $(\Xi,\gamma_N, Q_N)$-Wright-Fisher graph $(V_N,E_N)$.
\begin{defn}\label{def:wfb}
We say that $(g-r,l)\in V_N$ is a \emph{potential ancestor} of $(g,i)\in V_N$, $g\in \Z$, $r\in \N$ $l,i\in\{1,\ldots,N\}$, if there exists a path of $r$ connected vertices in $(V_N,E_N)$ that starts in $(g-r,l)$ and ends in $(g,i)$.\end{defn}
 For all $v\in V_N$, we define the following sets:
 \begin{itemize}
 \item \emph{Ancestors of an individual}: for every $v\in V_N$ $$A^N(v):=\{s\in V_N:s\text{ is an ancestor of }v\}.$$  
 \item    \emph{Ancestry of a sample}: for every $g\in\Z$, $n\leq N$ and every $v_1,\ldots,v_n\in V_N$ such that $g(v_1)=\ldots=g(v_n)=g$
 $$A^N (v_1,\ldots,v_n):=\bigcup_{i=1}^n A^N(v_i)$$
 \item \emph{Ancestors of a sample alive $r$ generations back in time}: for every $g\in\Z,r\in\N$, $n\leq N$ and every $v_1,\ldots,v_n\in V_N$ such that $g(v_1)=\ldots=g(v_n)=g$, 
 $$A^N_{(v_1,\ldots,v_n)}(r)=\{u\in A^N(v_1,\ldots,v_n):g(u)=g-r\}$$
 \end{itemize}
 
 Finally, let us write $|B|$ to denote the cardinality of a set $B$.
 \begin{defn}
The \emph{ancestral process}, or \emph{block-counting process}, of a sample of $n$ individuals $\mb v=(v_1,\ldots,v_n)$ from generation $g$ is the process $(D^N_r:r\in\N)$ counting the number of ancestors of the sample alive in each of the previous generations, i.e. $D^N_0=n$ and
 $$ D^N_r(\mb v)= |A^N_{(v_1,\ldots,v_n)}(r)|, \ \ r=1,2,\ldots.$$
 \end{defn}
Notice that the law of the process $(D^N_r)$ depends on the initial sample $\mb v$ only through the number $n$ of its coordinates. Our next goal is to prove a so-called \emph{sampling duality} property of $(\Xi,\gamma_N,Q_N)$-Wright-Fisher graphs, adapting the approach of \cite{M99} to graphical representations.
 \begin{prop}[Sampling duality]\label{pr:sdual}
Consider  a $(\Xi,\gamma_N,Q_N)$-Wright-Fisher graph $(V_N,E_N)$ defined on some probability space $(\Omega,{\cal F},\P)$. Let $(X_g^N:g\in\N)$ and $(D^N_r:r\in\N)$ be the corresponding 0-type allele frequency process and ancestral process, respectively. Consider the sampling probability function $S(x,n)$ defined in \eqref{sfunct}. For every $n,g\in\N$ and $x\in [N]/N,$
\begin{equation}
\E_x\left[S(X^N_g,n)\right]=\E_n\left[S(x,D_g^N)\right], \label{sdual}
\end{equation}
where the expectation on the left-hand side is on the random variable $X_g^N$ conditional on $X_0^N=x$ and the expectation on the right-hand side is on the variable $D_g^N$ conditional on $D^N_0=n$.
 \end{prop}
 \begin{rem}
In fact, Proposition \ref{pr:sdual} establishes a pathwise duality of $(X_g^N:g\in\N)$ and $(D^N_r:r\in\N)$ with respect to the duality function $S$, since the two processes are defined on the same probability space, both as functions of the same underlying driving process $(V_N,E_N)$ (see \cite{JK14} for a discussion of various definitions of duality).
 \end{rem}
 \begin{proof}
 Fix $m,n\in\N$ and a sample $\mb v_0=(v_1,\ldots,v_n)$ of size $n$ from generation $g+1$.
 Define the following events.
 
\begin{eqnarray*}
\overrightarrow{W}(i)&=&\{\text{There are }i\text{ individuals of type }0\text{ in generation } g\}=\{X^N_{g}=i/N\};\\
&&\ \\
\overleftarrow{W}(i)&=&\{\text{There are }i\text{ ancestors of }\mb v_0 \text{ in generation }1\}=\{D^N_g(\mb v_0)=i\}.
\end{eqnarray*}
 Finally, define the event
\begin{equation}
\mathcal{E}=\{\text{All individuals in }\mb v_0 \text{ are of type } 0\}.
\end{equation}

Note that the events $\overrightarrow{W}(i)$, $\overleftarrow{W}(i)$ and $\mathcal{E}$ belong to the $\sigma$-algebra generated by the Wright Fisher graph $(V_N,E_N)$.
We can use the law of total probabilities in two different ways to calculate the probability of $\mathcal{E}$, conditional on $\{X^N_{0}=m/N\}$. On one hand we have

\begin{eqnarray}\label{eq:eventE}
&&\P_{\{X^N_{0}=m/N\}}(\mathcal{E}  ) =\sum_{i=0}^N\P_{\{X^N_{0}=m/N\}}\left(\mathcal{E}\,|\,\ \overrightarrow{W}(i)\right)
\P_{\{X^N_{0}=m/N\}}\left(\overrightarrow{W}(i)\right)\ \notag \\
&&\ \ =\sum_{i=0}^N\Prob{\mathcal{E}\,|\,X^N_{g}=\frac{i}{N}, X^N_{0}=\frac{m}{N}}\Prob{X^N_{g}=\frac{i}{N}\mid\  X^{N}_{0}=\frac{m}{N}}\notag \\
&&\ \ =\sum_{i=0}^N S(i/N,n)\P_{m/N}(X^N_g=i/N)\notag\\
&&\ \ =\E_{m/N}[S(X^N_g,n)].
\end{eqnarray}
The third equality follows from the Markov property of $(X^N)$ and Proposition \ref{pr:sfunct}: if there are $i$ type zero individuals in generation $g$, the event $\mathcal{E}$ happens with probability $S(i/N,n)$. \\


Similarly, we can calculate the probability of $\mathcal{E}$ by conditioning on the number of ancestors of $\mb v_0$ at generation $1$.
\begin{eqnarray}\label{eq:eventE2}
&&\P_{\{X^N_{0}=m/N\}}(\mathcal{E} )\notag\\
&&\ \ \ =\sum_{i=0}^{N}\P_{\{X^N_{0}=m/N\}}\left({\mathcal E\,|\,\overleftarrow{W}(i)}\right)
\P_{\{X^N_{0}=m/N\}}\left(\overleftarrow{W}(i)\right)\notag\\
&&\ \ \ =\sum_{i=0}^{N}S\Big(\frac{m}{N},i\Big) \P_n\left(D^N_g(\mb v_0)=i\right) 
\notag\\
&&\ \ \ = \E_{n}\left[S\left(\frac{m}{N}, D^N_g\right)\right].
\end{eqnarray}
Finally notice that all the above probabilities vanish for $m=0$, or are conditioning on zero-probability events (and can arbitrarily be set to zero).
 \end{proof}
In some cases, sampling duality is very close to moment duality. 
\begin{corollary}\label{momentandsampling}
Assume that $Q_N(K_v=1)=1-\rho_N,$ $0\leq\rho_N\leq 1$. Then
$$
\E_x[(X^N_g)^n]=\E_n[x^{D^N_g}]+O(\rho_N)+O(\gamma_N).
$$
\end{corollary}
\begin{proof}
The proof follows immediately from Proposition \ref{pr:sfunct}, Proposition \ref{pr:sdual} and from the fact that, if $\P(K_v=1)=1-\rho_N$, the pgf of $Q_N$  satisfies
\begin{equation}
\varphi_{Q_N}(x)\approx(1-\rho_N)x+\rho_N  \E_{Q_N}\left[x^{K_v}\mid K_v>1\right].\label{asympgf}
\end{equation}
\end{proof}

\section{Convergence}\label{sec:SDE}

Now we will focus on the case $Q_N(K_v=1)=1-\rho_N\to 1$ as $N\to\infty$ and show weak convergence of the process $X^N_{\lfloor t/\rho_N\rfloor}.$

\begin{defn}\label{xialpha}
For any finite measure $\Xi$ on $\nabla_\infty$ and any $\alpha\in(0,1/2)$, define $$\Xi^\alpha_N((z_i)_1^\infty\in\cdot)=\frac{\Xi((z_i)_1^\infty\in\cdot)}{\sum_{i=1}^\infty z_i^2}1_{\{z_1\geq N^{-\alpha}\}}.$$ Furthermore, denote $\widehat\Xi:=\Xi/\Xi(\nabla_\infty)$ and $\widehat\Xi^\alpha_N:=\Xi^\alpha_N/\Xi^\alpha_N(\nabla_\infty)$.
\end{defn}
\begin{rem}\label{aproxi}
Note that $\Xi^\alpha_N(\nabla_\infty)\leq \Xi(\nabla_\infty)N^{2\alpha}$. In particular $\Xi^\alpha_N(\nabla_\infty)/N^\beta\rightarrow 0$ for all $\beta>2\alpha$. \\
The following Lemma will be useful.
\begin{lemma}
\label{l:xin}
If $Z$ has distribution $\widehat \Xi$,  $Z^{(N)}$ has distribution $\widehat \Xi^\alpha_N$ and $(B_i:i\in\N)$ is an \emph{i.i.d.} sequence of Bernoulli ($x$) random variables, independent from $Z$ and $Z^{(N)}$, then
\begin{eqnarray}\label{xin}
&&\Xi(\nabla_\infty)\Ex{\left(\sum_{i=1}^\infty\frac{(B_i-x)Z_i}{\sum_{i=1}^\infty Z_i^2}\right)^2}-\Xi^\alpha_N(\nabla_\infty)\Ex{\left({\sum_{i=1}^\infty(B_i-x)Z^{(N)}_i}\right)^2}\notag\\
&&\ \ \ \ =x(1-x)\int_{\nabla_\infty}\mbb{I}_{\{z_1< N^{-\alpha}\}}(z)\ \Xi(d z)\rightarrow 0,\ \ \ \ \ \text{as}\ N\to\infty.\notag\\
\end{eqnarray}
\end{lemma}
\begin{proof}
The proof is immediate by independence of the Bernoulli random variables $(B_i)$, and from the fact that 
$$\Xi(\nabla_\infty)-\Xi_N^\alpha(\nabla_\infty)\Ex{\sum_{i=1}^\infty {Z^{(N)}_i}^2}=\int_{\nabla_\infty}\left[1-\mbb{I}_{\{z_1\geq N^{-\alpha}\}}(z)\right]\ \Xi(dz).$$
\end{proof}
\end{rem}

\begin{prop}\label{lemma1}
Fix $\Xi$ a finite measure in $\nabla_{\infty}$ and let $\Xi^\alpha_N,\widehat\Xi^\alpha_N$ be as in Definition \ref{xialpha} for some $\alpha<1/2$. Let $(X^N_g)$ be the frequency process associated to a Wright-Fisher graph with parameters $(\widehat\Xi^\alpha_N, \gamma_N,Q_N)$, for some $\alpha<1/2$,  and suppose that there exist $\kappa\in(0,\infty)$  and $\sigma\geq 0$ such that
\begin{itemize}
\item[(i)] $\gamma_N=\kappa^{-1}\Xi^\alpha_N(\nabla_\infty)\times  \rho_N+o(\Xi^\alpha_N(\nabla_\infty)\times  \rho_N)$;
\item[(ii)] $\lim_{N\rightarrow \infty} 1/(N\rho_N)= \sigma/\kappa <\infty$ and $\rho_N N^{2\alpha}\rightarrow 0$;
\item[(iii)] $\lim_{N\rightarrow \infty}Q_N(K= k\mid K>1)=\pi_{k-1}$ for every $k\geq 2$;
\item[(iv)] $\beta:=\lim_{N\to\infty}\E_{Q_N}\left[K-1\mid K>1\right]=\sum_{k=1}^\infty k\pi_{k}<\infty$;
\end{itemize}
Then  $\Big(X_{\lfloor \kappa t/\rho_N\rfloor}^N\Big)\Rightarrow (X_t)$, where $(X_t)$ is the unique strong solution to the SDE 
\begin{eqnarray}\label{dualcatsde}
d X_t&=&- \kappa s(X_t)X_t(1-X_t)dt+\sqrt{ \sigma X_t(1-X_t)}dB_t\notag\\&&+\int_{(0,1]}\int_{(0,1]}\sum_{i=1}^\infty y_i\Big(\mathbf{1}_{\{u_i\le X_{t-} \}}-X_{t-}\Big) \widetilde{N}(d t,d y,d \overline u),
\end{eqnarray}
where $s(x)= \sum_{k=1}^\infty \P(K^*\geq k)x^{k-1}$, $K^*$ has distribution $\P(K^*=j)=\pi_j$ for all $k\in\N$,  $\widetilde{N}$ is a compensated Poisson measure on  $(0,\infty)\times \nabla_{\infty}\times [0,1]^\infty$ with intensity $ds\times\frac{\Xi(dy)}{\sum_{i=1}^\infty y_i^2}\times d\overline{u}$, where $d\overline{u}$ is the Lebesgue measure on $[0,1]^\infty$. 
\end{prop}

\begin{rem}\label{r:gen}
If $X_t$ exists,  it has generator $A$ with domain $D(A)$ containing all the twice differentiable functions $f:[0,1]\rightarrow \R$, such that, for every $x\in [0,1],$
\begin{eqnarray}
\label{dualc}
Af(x)&=& \kappa\sum_{i=1}^\infty \pi_i(x^{i+1}-x)f'(t)+ \frac{\sigma}{2} x(1-x)f''(x)\nonumber\\&&+\int_{\nabla_\infty}\Ex{{f(x(1-\sum_{i=1}^\infty z_i)+\sum_{i=1}^\infty z_iB_i)-f(x)}}\frac{\Xi(dz)}{\sum_{i=1}^\infty z_i^2},
\end{eqnarray}
for a collection $(B_i)$ if \emph{i.i.d.} Bernoulli $(x)$ random variables. This follows from the fact that 
\begin{eqnarray*}
\sum_{i=1}^\infty \pi_i(x^{i+1}-x)&=&-x(1-x)\sum_{i=1}^\infty \pi_i\frac{1-x^{i}}{1-x}\\
&=&-x(1-x)\sum_{i=1}^\infty \pi_i\sum_{j=0}^{i-1}x^j\\
&=&-x(1-x)\sum_{j=0}^\infty x^j \sum_{i=j+1}^{\infty}\pi_i\\
&=&- x(1-x)\sum_{j=0}^\infty x^j \sum_{i=j+1}^{\infty} \P(K^*= i)\\
&=&- x(1-x) \sum_{j=1}^\infty \P(K^*\geq j)x^{j-1}\\&=&
-x(1-x)s(x)
\end{eqnarray*}
which accounts from the first sum in \eqref{dualc}. The remaining terms, and in particular the integral, follow from a simple reformulation of the generator of the two-type $\Xi$-Fleming-Viot process without selection (see e.g. formula (5.6) in \cite{BBetal09}), but could also be derived directly from \eqref{dualcatsde} by Poisson calculus techniques. 

In view of this, the drift coefficient $-s(x)x(1-x)$ can be also be written as
\begin{equation}
-x(1-x)s(x)= [\varphi_{\pi}(x)-x]
\label{bpdrift}
\end{equation}
where $\varphi_{\pi}$ is the probability generating function of $(\pi:i>1,\ldots),$ the distribution of $K^*$. 
The form \eqref{bpdrift} highlights an important connection between our generator $A$ and that of a class of Branching-coalescing stochastic processes with $K^*$ denoting the offspring random variable of the branching component. This connection will be explored further in Section \ref{sec:cbp}.
\end{rem}

Before starting with the proof of Proposition \ref{lemma1}, we shall make sure that a solution $(X_t)$ to the SDE \eqref{dualcatsde} in fact exists. This is the content of the following 
\begin{lemma}
\label{l:exist}
For any $\sigma\geq 0,\kappa>0$, any finite measure $\Xi$ on $\nabla_\infty$ and any probability mass function $\{\pi_i:i\in\N\}$ with finite mean, there exists a unique strong solution $(X_t)$ to the SDE \eqref{dualcatsde}, associated to the generator $A$ as in Equation \eqref{dualc}.
\end{lemma}

\begin{proof}
To prove strong  existence and pathways uniqueness of $X^N$, we will use Theorem 5.1 of Li and Pu (2012) \cite{LP12}.  In particular, we need to verify the sufficient conditions 3.a, 3.b and 5.a of that paper. 3.a in our case amounts simply to proving Lipschitz continuity for the drift coefficient, which is verified by observing that 

\begin{equation}
\Big\vert\sum_{k=1}^\infty{\pi_{k} (x^{k+1}-x)}-\sum_{k=1}^\infty \pi_{k} (y^{k+1}-y)\Big\vert<\max\{1,\beta\}|x-y|.\tag{3.a}
\end{equation}
This follows from the fundamental Theorem of calculus since, if we denote $u(x)={\sum_{k=1}^\infty \pi_{k} (x^{k+1}-x)}$, then $-1\leq u'(x)\leq\sum_{k=1}^\infty k\pi_{k} =\beta$ for $x\in[0,1]$.
\\
To prove condition 3.b, define $g(x,\overline{u}\mid z)=\sum_{i=1}^\infty z_i (\mbb{I}_{\{u_i<x\}}-x)$ for any $(x,\overline u, z)\in[0,1]\times [0,1]^\infty\times\nabla_\infty$ and notice that
$$
\int_{\nabla_\infty}\int_{[0,1]^\infty} g(x,\overline u\mid z)\ d\overline u\  \frac{\Xi(dz)}{\sum_{j=1}^\infty z_i^2} =\int_{\nabla_\infty}\sum_{i=1}^\infty z_i\Ex{B^x_i-x} \frac{\Xi(dz)}{\sum_{j=1}^\infty z_i^2}, 
$$
where $(B_i^x:i\in\N)$ is an \emph{i.i.d.} sequence of Bernoulli $(x)$ random variables (in other words, $(B_i^x)=(\mbb{I}_{\{U_i\leq x\}})$ for $(U_i)$ \emph{i.i.d.} Uniform in $[0,1]$). Our next task is to show that there exist a constant $K$ such that, for every $x,y$ where $0\leq x, y \leq 1,$
\begin{align}
&\,|\sqrt{\sigma x(1-x)}-\sqrt{\sigma y(1-y)}|^2\notag\\
&\,\hspace{1cm}+\int_{\nabla_\infty}\int_{[0,1]^\infty} (g(x,\overline u\mid z)-g(y,\overline u\mid z))^2\ d\overline u\  \frac{\Xi(dz)}{\sum_{j=1}^\infty z_i^2} <K|x-y|.\notag\\
&\, 
\tag{3.b}
\end{align}
To this purpose, we first claim that
\begin{equation}
|\sqrt{\sigma x(1-x)}-\sqrt{\sigma y(1-y)}|^2\leq \max\{4,2\sqrt{\sigma}\} |x-y|.\label{pro3b}
\end{equation}
Note that the claim is trivial if $|x-y|\geq 1/4$.  Now assume that  $|x-y|< 1/4$. There are two cases: First assume that $x,y\in[1/4,1]$  and note that $|\sqrt{\sigma x(1-x)}-\sqrt{\sigma y(1-y)}|^2<|\sqrt{\sigma x}-\sqrt{\sigma y}|$. Let $f(x):=\sqrt{\sigma x}$ and note that $|f'(x)|<2\sqrt{\sigma}$ for all $x\in[1/4,1]$. Then, by the fundamental theorem of calculus
$$|\sqrt{\sigma x(1-x)}-\sqrt{\sigma y(1-y)}|^2<|\sqrt{\sigma x}-\sqrt{\sigma y}|=\left\lvert\int_y^xf'(s)ds\right\rvert<2\sqrt{\sigma}|x-y|.$$
Now assume that $x,y\in[0,3/4]$. Then \eqref{pro3b} is proved by a similar argument, by using $|\sqrt{\sigma x(1-x)}-\sqrt{\sigma y(1-y)}|^2<|\sqrt{\sigma(1- x)}-\sqrt{\sigma(1- y)}|$ and with $f(x)$ replaced by $h(x):=\sqrt{\sigma(1- x)},$ so that  $|h'(x)|<2\sqrt{\sigma}$ for all $x\in[0,3/4]$. \\
Finally, assuming without loss of generality $x>y$,
\begin{eqnarray*}
&& \int_{\nabla_\infty}\int_{[0,1]^\infty} (g(x,\overline u\mid z)-g(y,\overline u\mid z))^2\ d\overline u\  \frac{\Xi(dz)}{\sum_{j=1}^\infty z_i^2}\nonumber\\
&& \ \ \ \ =\int_{\nabla_\infty}\Ex{\sum_{i=1}^\infty z_i(B^x_i-x- B^y_i+y))^2}\ \frac{\Xi(dz)}{\sum_{j=1}^\infty z_i^2}\nonumber \\
&& \ \ \ \ =\int_{\nabla_\infty}\sum_{i=1}^\infty z^2_i\Ex{[(B^x_i- B^y_i)+(y-x)]^2}\ \frac{\Xi(dz)}{\sum_{j=1}^\infty z_i^2}\nonumber \\
&& \ \ \ \ =\int_{\nabla_\infty}\sum_{i=1}^\infty z_i^2\left((x-y)(1-(x-y))\right)\ \frac{\Xi(dz)}{\sum_{j=1}^\infty z_i^2}\nonumber \\
&& \ \ \ \  = \int_{\nabla_\infty}(x-y)(1-(x-y))\ \Xi(dz)\nonumber \\
&& \ \ \ \ <  \ \Xi({\nabla_\infty}) |y-x|,
\end{eqnarray*}
where the third equality follows from the fact that each difference $(B^x_i-B^y_i)=\mbb{I}_{\{U_i\leq x\}}-\mbb{I}_{\{U_i\leq y\}}$ is Bernoulli$(x-y)$. This proves condition 3.b and we are left with the task of proving condition 5.a., i.e. that there is a constant $M$ such that
\begin{equation}
s^2(x)+\sigma x(1-x)+\int_{[0,1]^\infty\times\nabla_\infty}g^2(x,\overline u\mid z)\ d\overline u\frac{\Xi(dz)}{\sum_{j=1}^\infty z_i^2}<M,
\tag{5.a}
\end{equation}
for every $x\in[0,1]$.
However this is easy, after the above calculations. Indeed we can rewrite the left hand side as
\begin{eqnarray*}
&&\left(\sum_{i=2}^\infty \pi_i (x^{i}-x)\right)^2+\sigma x(1-x)+
\int_{\nabla_\infty}\Ex{\sum_{i=1}^\infty z_i^2(B^x_i-x)^2}\frac{\Xi(dz)}{\sum_{j=1}^\infty z_i^2}\\
&&\ \ \  \ \ \ \ \ \ \leq \beta^2+\sigma+\Xi(\nabla_\infty).
\end{eqnarray*}
This implies that all the sufficient conditions of Li and Pu's theorem are satisfied and we can conclude strong existence and uniqueness of $X_t$.\\
\end{proof}

We can 
turn our attention to the proof of our main convergence result.

\begin{proof}[Proof of Proposition \ref{lemma1}]
We will prove the convergence of the generator of $(X^N_{M_t})$, where $M^N_t$ is a Poisson process with rate $\kappa/\rho_N$, to the generator of $X_t$. Provided this claim is true, we can use Theorem 19.25 and Theorem 19.28 of \cite{K97}  to conclude that $(X^N_{\lfloor \kappa t/\rho_N\rfloor})$ converges weakly to  $(X^N_{t})$.\\
First of all, assumption (i) implies that, in proving convergence, we can work with the simplification $\gamma_N=\eta_N\rho_N/\kappa,$ where $\eta_N:=\Xi^\alpha_N(\nabla_\infty)$, without loss of generality. Notice that the construction requires $\gamma_N\in[0,1]$, but this, by assumption \emph{(ii)}, is always satisfied at least for $N$ sufficiently large.\\
From formula \eqref{xi_tf} of Corollary \ref{cor:tf}, it is convenient to have in mind the representation:
\begin{equation}
X^N_{g}\mid \{X_{g-1}=x\}=(1-H_g)\frac{M_x}{N}+H_g \frac{M_{Y(x)}}{N},
\label{xi_deq}
\end{equation}
where:  $H_g$ has Bernoulli$(\gamma_N)$ distribution; $M_{Y(x)}$ (resp. $M_x$) is a binomial random variable with parameter $(N,\varphi_{Q_N}(Y(x)))$ (resp. $(N,\varphi_{Q_N}(x))$), with $Y(x)$ described by \eqref{Y(x)} and, given $X_{g-1}=x$, all the random variables on the right-hand side are conditionally independent.\\
Furthermore, the assumption that $Q_N(\{1\})=1-O(\rho_N)$ entails \eqref{asympgf} i.e.
$\varphi_{Q_N}(x)\approx x$ as $ N\to\infty,\rho_N\to 0,$ hence $M_{Y(x)}$ is, asymptotically, binomial with parameter $(N,Y(x))$. By a routine use of the binomial theorem, from the definition \eqref{Y(x)} of $Y(x),$ conditionally on $Z=z$ we can thus write
\begin{equation*}
M_{Y(x)}\mid \{Z=z\}\approx \frac{\sum_{i=1}^\infty B_i\kappa_i+R }{N},\ \ \ \ \ \ N\to\infty,\rho_N\to 0,
\end{equation*}
where $(B_i)$ are \emph{i.i.d.} Bernoulli$(x)$, $R$ is binomial$(\kappa_0, x)$ and $(\kappa_0,\kappa_1,\ldots)$ is a sequence of non-negative integers with multinomial distribution with parameter $(N,(1-|z|, z_1,\ldots)).$
For every $g$, we can write 
$$\Ex{f(X^{N}_g)\mid X^N_{g-1}=x,H_g=1,Z_g=z}\approx \Ex{f\left(\frac{\sum_{i=1}^\infty B_i\kappa_i+R }{N} \right)}$$

Now we consider the following Poisson embedding of the $0$-type frequency process $(X^N_g)$: Let $M_t$ be a Poisson process on $[0,\infty)$ with rate $(\kappa/\rho_N),$ independent of $(X^N_g)$, and define $\wt{X}^N_t:= X^N_{\lfloor M_t\rfloor}$ for every $t\geq 0$. We can see that the discrete generator $A^N$ of $(\wt{X}^N_t)$, applied to any function $f\in C^2[0,1]$ in a point $x\in[0,1]$, fulfils
\begin{eqnarray}
A^Nf(x)&:=&\lim_{t\rightarrow 0} \frac{\E_x\left[f(\wt X^N_{\lfloor \kappa t/\rho_N \rfloor})\right]-f(x)}{ t}\nonumber\\
&=& \eta_N \rho_Nk^{-1}\frac{\Ex{f\left({M_{Y(x)}}/{N}\right)-f(x)}}{\rho_N\kappa^{-1}} + (1-\eta_N \rho_N\kappa^{-1})\frac{\kappa\Ex{f\left({M_{x}}/{N}\right)-f(x)}}{\rho_N}\nonumber\\
&&\ \notag\\
&\approx&\eta_N\int_{\nabla_\infty}\Ex{f\left(\frac{\sum_{i=1}^\infty B_i\kappa_i+R }{N}\right)-f(x)\mid Z=z}\ \widehat\Xi^\alpha_N(dz)\nonumber\\
&&\ \notag\\
&&+(1-\eta_N \rho_N\kappa^{-1})\frac{\kappa\E_x[f(M_x/N)-f(x)]}{\rho_N}.\nonumber\\
\end{eqnarray}
Add and subtract $x$ to the first term and Taylor-expand the second term around $x$ to obtain
\begin{eqnarray}
&\approx&\eta_N\int_{\nabla_\infty}\Ex{f\left(x+\frac{\sum_{i=1}^\infty (B_i-x)\kappa_i+R -{\kappa_0x}}{N}\right)-f(x)\mid Z=z}\ \widehat\Xi^\alpha_N(dz) \  \notag\\
&&\  \label{EQGEN1}\\
&&\ \notag\\
&&+(1-\eta_N \rho_N\kappa^{-1})\frac{\kappa\E_x[M_x/N-x]f'(x)}{\rho_N}\label{EQGEN2}\\
&&\ \notag\\&&+(1-\eta_N \rho_N\kappa^{-1}) \frac{\kappa\E_x[(M_x/N-x)^2]f''(x)}{{2}\rho_N}\label{EQGEN3}+o(1)
\end{eqnarray}
 Now we will study separately the parts \eqref{EQGEN1}, \eqref{EQGEN2}, \eqref{EQGEN3}. The simplest is \eqref{EQGEN3}, for which we just need to use that $Q_N(\{1\})=1-o(1)$ and write explicitly the variance of $M_x/N$. By assumption \emph{(ii)} and because $\eta_N \rho_N\to 0,$
\begin{eqnarray}
(1-\eta_N \rho_N/\kappa)\frac{\kappa\E_x[(M_x/N)-x)^2]f''(x)}{{2}\rho_N}=\frac{\sigma}{2} x(1-x)f''(x)+o(1).\label{EQGEN4}
\end{eqnarray}
Now we study Part \eqref{EQGEN2}.
\begin{eqnarray}
(1-\eta_N\rho_N/\kappa) \frac{\kappa\E_x[\frac{M_x}{N}-x]f'(x)}{\rho_N}&=&(1-\eta_N \rho_N/\kappa)\frac{\varphi_{Q_N}(x)-x}{\rho_N/\kappa}f'(x)\nonumber
\notag\\
&=&\kappa\sum_{k=2}^\infty (x^k-x)\frac{Q_N(\{k\})}{\rho_N}f'(x)+o(1)\nonumber\\
&=&\kappa\sum_{k=1}^\infty (x^{k+1}-x)\pi_k f'(x)+o(1)\notag\\
&=&  \kappa x(1-x)s(x)+o(1).\label{EQGEN5}
\end{eqnarray}
The last equality follows from Remark \ref{r:gen}. Finally we study part \eqref{EQGEN1}. First expand around $x+\sum_{i=1}^\infty(B_i-x)z_i$, conditioning on $\kappa, (B_i), R, z:$ 
\begin{eqnarray}
&&f\left(x+\frac{1}{N}\left[\sum_{i=1}^\infty(B_i-x)\kappa_i+R-\kappa_0x\right]\right)\notag\\
&&\ \ \ \ \ =f\left(x+\sum_{i=1}^\infty(B_i-x)z_i+\sum_{i=1}^\infty(B_i-x)(\kappa_i/N-z_i)+R/N-\kappa_0x{/N}\right)\nonumber\\
&&\ \ \ \ \ =f\left(x+\sum_{i=1}^\infty(B_i-x)z_i\right)+\frac{1}{2}\left(\sum_{i=1}^\infty(B_i-x)(\kappa_i/N-z_i)+\frac{1}{N}[R-\kappa_0x]\right)^2f''(\xi)\nonumber
\end{eqnarray}
where $\xi$ is a number between $x+\sum_{i=1}^m(B_i-x)z_i$ and $x+\sum_{i=1}^m(B_i-x)\kappa_i/N+ [R-\kappa_0x]/N$. Next note that, since $\Ex{R}=\kappa_0 x$ and $\Ex{\kappa_i}=Nz_i$ then, conditionally on $Z=z,$ 
\begin{eqnarray*}
&&
\E_z\left[\left(\sum_{i=1}^\infty(B_i-x)(\kappa_i/N-z_i)+\frac{1}{N}[R-\kappa_0x]\right)^2\right]\notag\\
&&\ \ \ \ =\ x(1-x)\E_z\left[\sum_{i=1}^\infty\frac{z_i(1-z_i)}{N}+\frac{\kappa_0}{N^2}\right]\notag\\
&&\ \ \ \ = \ x(1-x)\left[\sum_{i=1}^\infty\frac{z_i(1-z_i)}{N}+\frac{1-|z|}{N}\right]\leq \frac{2}{N}.
\end{eqnarray*}
This is useful to us because, still by Remark \ref{aproxi}, it implies that
\begin{eqnarray*}
\int_{\nabla_\infty}\E_z\left[\left(\sum_{i=1}^\infty(B_i-x)(\kappa_i/N-z_i)+\frac{1}{N}[R-\kappa_0x]\right)^2\right]\Xi^\alpha_N(dz)&\leq&\frac{2}{N}\ \Xi^\alpha_N(\nabla_\infty)\\
&\leq& 2N^{2\alpha-1} \  \Xi(\nabla_\infty).
\end{eqnarray*}

Recall that we assumed $N^{2\alpha-1}\rightarrow 0$, so we have shown that Part \eqref{EQGEN1} is equal to
\begin{eqnarray*}
&&
\eta_N\int_{\nabla_\infty}\E_z\left[{f(x+\sum_{i=1}^\infty(B_i-x)z_i)-f(x)}\right]\widehat\Xi^\alpha_N(dz)
+o(1)\notag\\
&&= \int_{\nabla_\infty} \E_z\left[{f(x+\sum_{i=1}^\infty(B_i-x)z_i)-f(x)}\right]{\Xi^\alpha_N(dz)}+o(1).
\end{eqnarray*}
Finally, by Lemma \ref{l:xin} this is also is equal to
\begin{eqnarray}\label{EQGEN6}
&&\int_{\nabla_\infty}\E_z\left[{f(x+\sum_{i=1}^\infty(B_i-x)z_i)-f(x)}\right]\frac{\Xi(dz)}{\sum_{i=1}^\infty z_i^2}+o(1).\notag\\
\end{eqnarray} 
Equations \eqref{EQGEN4}, \eqref{EQGEN5} and \eqref{EQGEN6} imply that $A^Nf \rightarrow Af$.
\end{proof}

\subsection{Duality for the limit genealogy}\label{sec:cbp}
Now we shall focus on the sampling dual of $(X^N_g)$, i.e. the ancestral process $(D_g^N)$ introduced in Section \ref{sec:a&d}. We will work with the inverse process of $(D_g^N)$, which we define as $(S_g^N)=(1/D_g^N)$, with the benefit of mapping the sequence of $\N$-valued processes $((S^N):N=1,2,\ldots)$ into a sequence of processes in the compact state space $[N]/N\subseteq [0,1]$. This will allow us to guarantee weak convergence of the processes just by exploiting weak convergence of their (semi)-duals, $((X^N):N=1,2,\ldots).$
\begin{defn}\label{intdual}
Fix $a,b\in\R,$ a probability mass function $p=(p_i:i\in\N)$  on $\N\cup\{\infty\}$ and a finite measure  $\Xi$ on $\nabla_{\infty}$. We call the $(a,p,b,\Xi)$-Branching-Coalescent process the continuous time Markov chain $(D_t)$ with state space $\N\cup\{\infty\}$ and generator
\begin{eqnarray}
Lf(n)=a\sum_{i=0}^\infty p_i[f(n+i-1)-f(n)] + b{n\choose 2}[f(n-1)-f(n)]+\overline{L}f(n)
\label{L_n}
\end{eqnarray}
for every $n\in\N$ and $f:\N\rightarrow \R$ in $C_2,$ with
 \begin{eqnarray}
 &&\overline{L}f(n):=\notag\\
 &&\int_{\nabla_\infty}\sum_{k=0}^n\sum_{\{\overline m: |\overline m|=k\}}\left[f(n-k+d(\overline m))-f(n)\right]\text{Bin}(k;n,|z|)\text{ Mn}(\overline m; k,z)\frac{\Xi(dz)}{\sum_{i=1}^\infty z_i^2},\notag\\
 &&
 \label{lbar}
\end{eqnarray}
where $\text{ Mn}(\cdot; k,z)$ is the multinomial probability mass function with parameter $(k,z)$, and $d(\overline{m})$ counts the number of non-zero coordinates in $\overline{m}=(m_1,m_2,\ldots)$.
\end{defn}

\begin{prop}\label{pr:limdual}
Assume the same notation and hypotheses of Proposition \ref{lemma1}. Let $(X^N_g:g\in\N)$ and $(D^N_g:g\in\N)$ be, respectively, the type-0 allele frequency process and the ancestral process of a $(\widehat\Xi^\alpha_N, Q_N,\gamma_N)$-Wright-Fisher graph $(V_N,E_N)$. Define $S^N_g:= 1/D^N_g$ for every $g\in\N$. Then $(S^N_{\lfloor t/\rho_N\rfloor}:t\geq 0)$ converges in distribution to $(1/D_t:t\geq 0)$, where $(D_t)$ is a $(\kappa,\pi,\sigma,\Xi)$-Branching-Coalescent process with generator $L$ as defined in \eqref{L_n}. Moreover,
 for every $x\in[0,1]$, $n\in \N$ and $t>0$
$$
\E_x[X_t^n]=\E_n[x^{D_t}].
$$
 where $(X_t)$ is the two-type $(\kappa,\pi,\sigma,\Xi)$-Fleming-Viot process solution to the SDE \eqref{dualcatsde} of Proposition \ref{lemma1}. 
  \end{prop}
\begin{proof}
It is easy to show directly that, if $A$ is the two-type $(\Pi,\sigma,\Xi)$-Fleming-Viot generator then, for $h(x,n)=x^n,$ $A h(x,n)=L h(x,n)$ where $L$ is as \eqref{L_n}, with $a=\kappa, b=\sigma, p=\pi$, although one could as well combine together results already proven in the literature for the cases where there are no simultaneous multiple coalescence events (\cite{GPP16}) or $\kappa=0$ (\cite{BBetal09}).
For every $x\in [0,1],$
\begin{eqnarray}
A x^n&=&\kappa \sum_{i=1}^\infty {\pi_i}[x^{n+i}-x^{n}]+\sigma {n\choose 2} (x^{n-1}-x^n)\label{fouc}\\
&&\ \ \ + \int_{\nabla_\infty}\left(\Ex{\Big(x(1-|z|)+\sum_{i=1}^\infty z_iB_i\Big)^n}-x^n\right)\frac{\Xi(dz)}{\sum_{i=1}^\infty z_i^2}.\label{nonfouc}
\end{eqnarray}
The first line, \eqref{fouc}, describes the action on $n\mapsto h(x,n)$ of the Markov generator $L-\bar L$. 
As for the $\bar L$-part, rewrite the expectation in the second term as 
\begin{eqnarray}
&&\Ex{\Big(x(1-|z|)+\sum_{i=1}^\infty z_iB_i\Big)^n}=\sum_{k=0}^n{n\choose k} x^{n-k}(1-|z|)^{n-k}\E{\Big(\sum_{i=1}^\infty z_iB_i\Big)^{k}}\notag\\
\end{eqnarray}
Since $(B_i)$ is an \emph{i.i.d.} sequence of Bernoulli $(x)$ random variables,
$$\E{\Big(\sum_{i=1}^\infty z_iB_i\Big)^{k}}=|z|^k\ \sum_{|\overline m|=k} \text{Mn}(\overline m; k, z/|z|)x^{k-d(\overline m)},$$
where $d(\overline m)$ counts the non-zero coordinates of the vector $\overline m$. Hence \eqref{nonfouc} is equal to
\begin{eqnarray}
&&\ \ \ \int_{\nabla_\infty}\left[\sum_{k=0}^n\sum_{|\overline m|=k}{n\choose k}\text{Mn}(\overline m; k, z/|z|) |z|^k(1-|z|)^{n-k}x^{n-k+d(\overline m)}-x^n\right]\frac{\Xi(dz)}{\sum_{i=1}^\infty z_i^2}\notag\\
&&\ \ \ \ = \bar L x^n.\notag
\end{eqnarray}

Now, by Corollary  \ref{momentandsampling}, we have that 
$$\E_n[x^{D^N_t}]=\E_x[(X^N_t)^n]+O(\rho_N).$$
On the other hand, by Proposition \ref{lemma1}, for all $x\in[0,1]$, $n\in \N$ and $t>0$
$$
    \E_n[x^{N_t}]=\E_x[X_t^n]=\lim_{N\rightarrow \infty}\E_x[(X^N_{\lfloor t/\rho_N\rfloor})^n]=\lim_{N\rightarrow \infty}\E_n[x^{D^N_{\lfloor t/\rho_N\rfloor}}].
$$
Since convergence of pgfs implies convergence of the corresponding distributions, we conclude the convergence of the semigroup $p_t^N(f(n))=\E_n[f(D^N_{\lfloor t/\rho_n\rfloor})]$ to the semigroup $p_t(f(n))=\E_n[f(D_t)]$. Using composition of functions, this also implies that $\overline{p}_t^N(f(n))=\E_n[f((D^N_t)^{-1})]$ converges to $\overline{p}_t^N(f(n))=\E_n[f((N_t)^{-1})]$.  As $S_g^N=(D_g^N)^{-1}$ takes values in the compact $[0,1]$ we conclude that $(S_{\lfloor t/\rho_N\rfloor}^N)\Rightarrow (D_t)^{-1}$ by applying Theorem 19.25 and Theorem 19.28 of \cite{K97}.
\end{proof}

\section{Fixation in the ancestral process}
\label{sec:fixation} We will now turn our attention to studying the probability of extinction of the $0$-type allele in the jump-diffusion scaling limit $(X_t)$ with generator \eqref{dualc}. Such a probability is intrinsically related to the long-term behaviour of the dual ancestral process $(D_t)$. We have seen that, backward in time, selection is responsible for positive jumps  in $(D_t)$, due to the choice of multiple potential parents, whereas reproduction induces negative jumps, due to coalescence of lineages. We will determine, in Theorem \ref{THM}, conditions in order for none of these two forces to dominate the other so that $(D_t)$ has a stationary distribution. We shall see that the $0$-type has a chance to survive (and even fixate in the population) whenever $(D_t)$ has a stationary distribution. In fact, the stationary distribution of $(D_t)$ characterises the probability of fixation of type 1. Otherwise, if $(D_t)$ does get absorbed at one, then type $0$ is doomed to extinction: this will be the content of Lemma \ref{lemma:selection21}. 
\\
As motivated in the introduction, our focus in this Section is on population models without Kingman component, i.e. whose allele frequency process has no diffusive part, since this is the case that cannot be covered by other duality-based methods proposed recently (\cite{GPP16}). Thus we will assume throughout that $\Xi$ has no atoms at zero (equivalently, that $\sigma=0$). It is worth pointing out that our generator approach (in particular, the identity \eqref{GenFVbernulli}) relies on such an assumption.\\
Recall the definition $\widehat\Xi(dz):={\Xi(dz)}/{\Xi(\nabla_\infty)}.$
The following identity will play a key role.
\begin{lemma}\label{l:xibob}
Let $\Xi$ be a finite measure on $(\nabla_\infty)$ with no atoms at $\{0\}$. The generator $A$ of the two types $(\pi,0,\Xi)$ Fleming Viot process applied to a bounded function $f:[0,1]\rightarrow \R$ admits the representation
\begin{eqnarray}
&&Af(x)=-\kappa s(x) x(1-x) f^{\prime}(x)\notag\\
&& + \frac{\Xi(\nabla_\infty)}{2}\Ex{\frac{\left(-x+\sum_{i=1}^\infty Z^*_iB_i\right)\left(\sum_{i=1}^\infty Z^*_iB_i\right)}{(\sum_{i=1}^\infty {Z^*_i}^2)}f''\left(x(1-W)+VW\sum_{i=1}^\infty Z^*_iB_i\right)}\notag\\
\label{GenFVbernulli}\end{eqnarray}
 where $Z^*_i:={Z_i}/{|Z|}$ $i=1,2,\ldots$,  $Z=(Z_1,Z_2,...)$ is $\widehat\Xi$-distributed, $(B_i)_{i\in \N}$ is a sequence of \emph{i.i.d.} Bernoulli$(x)$- distributed random variables, $V$ is a uniform $[0,1]$ random variable , $W=|Z|S$, $S$ has density $2s$ on $[0,1]$, and $Z,S,V, (B_i)$ are independent.

\end{lemma}
\begin{rem}
The identity \eqref{GenFVbernulli} extends a representation for the generator of a two-type $\Lambda$-Fleming-Viot process proved by Griffiths (equation (7) in \cite{G14}), which can be recovered as the particular case where $Z_2=Z_3=\cdots=0$ with $\widehat\Xi$-probability one. In this case \eqref{GenFVbernulli} holds with $\sum_{i=1}^\infty Z^*_iB_i=B_1$.
\end{rem}

\begin{proof}
The drift component plays no substantial role in the proof so we can as well set $s(x)\equiv 0$ for convenience and we only need to calculate directly the expectation (integrating) with respect to $V$ and $S$ and apply a simple change of variables. With $s\equiv 0,$ denote with $A^*f$ the right-hand side of \eqref{GenFVbernulli} and with $Af$ the two type $\Xi$-Fleming-Viot generator. We will first calculate the expectation with respect to $V$. Integrating by parts,
\begin{align}
\frac{A^*f(x)}{\Xi(\nabla_\infty)}=&\frac{1}{2}\Ex{\frac{1}{\sum_{i=1}^\infty {Z^*_i}^2}\Big(-x+\sum_{i=1}^\infty Z^*_iB_i\Big)\Big(\sum_{i=1}^\infty Z^*_iB_i\Big)f''\Big(x(1-W)+VW\sum_{i=1}^\infty Z^*_iB_i\big)}\nonumber\\
=&\frac{1}{2}\Ex{\frac{1}{\sum_{i=1}^\infty {Z^*_i}^2}\Big(-x+\sum_{i=1}^\infty Z^*_iB_i\Big)\frac{1}{W}f'\Big(x(1-W)+W\sum_{i=1}^\infty Z^*_iB_i\Big)}\\
&-\frac{1}{2}\Ex{\frac{1}{\sum_{i=1}^\infty {Z^*_i}^2}\Big(-x+\sum_{i=1}^\infty Z^*_iB_i\Big)\frac{1}{W}f'(x(1-W))}.\label{V2}
\end{align} 
The last term, \eqref{V2}, in fact vanishes. Indeed, for any $z^*\in\nabla_\infty:|z^*|=1,$ $$\Ex{\sum_{i=1}^\infty z^*_iB_i}=\sum_{i=1}^\infty z^*_i\E[B_i]=x.$$ Thus
\begin{align*}
&\frac{1}{2}\Ex{\frac{1}{\sum_{i=1}^\infty {Z^*_i}^2}\Big(-x+\sum_{i=1}^\infty Z^*_iB_i\Big)\frac{1}{W}f'(x(1-W))}=\\
&\frac{1}{2}\Ex{\left(-x+\Ex{\sum_{i=1}^\infty Z^*_iB_i\mid Z^*}\right)\frac{f'(x(1-W)}{W{\sum_{i=1}^\infty {Z^*_i}^2}}}=0.
\end{align*} 
Finally, we calculate the expectation with respect to $S$.
\begin{align}
\frac{A^*f(x)}{\Xi(\nabla_\infty)}=&\frac{1}{2}\Ex{\frac{1}{\sum_{i=1}^\infty {Z^*_i}^2}\Big(-x+\sum_{i=1}^\infty Z^*_iB_i\Big)\frac{1}{W}f'\Big(x(1-W)+W\sum_{i=1}^\infty Z^*_iB_i\Big)}\\
=&\frac{1}{2}\Ex{\int_0^1\frac{1}{\sum_{i=1}^\infty {Z^*_i}^2}\Big(-x+\sum_{i=1}^\infty Z^*_iB_i\Big)\frac{1}{s|Z|}f'(x(1-s|Z|)+s|Z|\sum_{i=1}^\infty Z^*_iB_i)2sds}\nonumber\\
=&\Ex{\int_0^1\frac{1}{|Z| \sum_{i=1}^\infty { Z^*_i}^2}\Big(-x+\sum_{i=1}^\infty Z^*_iB_i\Big)f'\left(x(1-s|Z|)+s|Z|\sum_{i=1}^\infty Z^*_iB_i\right)ds}\nonumber\\
=&\int_{\nabla_\infty}\frac{1}{\sum_{i=1}^\infty {z_i}^2} \left[f\Big(x(1-|z|)+\sum_{i=1}^\infty z_iB_i\Big)-f(x)\right]\widehat \Xi(dz)\nonumber\\
=&\frac{A f(x)}{\Xi(\nabla_\infty)},\notag
\end{align} 
where the second-to-last equality follows from integration by parts and the last equality follows from \eqref{dualc}.
\end{proof}

Our main result on fixation will consider the dynamics of Branching-coalescing dual processes driven by a certain class of \emph{admissible} measures $\Xi$.
\begin{defn}
\label{def:admiss}
Let $z\in\nabla_{\infty}$ and $c\in(0,1)$. Define $m(z,c)=\inf\{k\in\N:\sum_{i=1}^kz_i>|z|(1-c)\}$. We say that $z$ is admissible if $$\lim_{n\rightarrow \infty}m(z,c_n)/\sqrt{n}=0,$$ for some sequence $c_n$ such that $n c_n\rightarrow 0$. We denote $\nabla_{\infty}^\circ$ the set of elements of $\nabla_{\infty}$ which are admissible. We say that a probability measure $\mu$ on $\nabla_\infty$ is admissible if $\mu(\nabla_\infty^\circ)=1.$
\end{defn}
\begin{example}[Finite support]
Every $\Xi$ measure with support in $\nabla_m:=\{z=(z_1,z_2,\ldots)\in\nabla_\infty: z_j=0\ \forall j>m\}$, for any $m\in\N$, is admissible. The parameter measure $\Lambda$ of a $\Lambda$-Fleming-Viot model is thus always admissible since $\Lambda$ is a $\Xi$-measure concentrated on $\nabla_1=[0,1]$.
\end{example}
\begin{example}[Stick-Breaking distributions]
Let $\{Y_n\}_{n\in \N}$ be a sequence of independent and identically distributed $[0,1)$ valued random variables, such that $\P(Y_1>0)>0$. Let $\bar Z_1=Y_1$ and $\bar Z_n=Y_n\prod_{i=1}^{n-1}(1-Y_i)$, $n=2,3,\ldots$. Let the random vector $(Z_1,Z_2,...)$ be a permutation of $(\bar Z_1,\bar Z_2,...)$ such that $Z_n>Z_{n+1}$ for all $n\in \N$. If $\Xi$ is the distribution of $( Z_1, Z_2,...)$ then $\Xi$ is admissible. To see this, let $\epsilon>0$ be such that $\delta:=\P(Y_1>\epsilon)>0$ and let $C_n=\sum_{i=1}^n \mbb{I}_{\{Y_i>\epsilon\}}$. We will show that $\P (m(Z,n^{-1/2})> n^{1/4})$ is exponentially small. That $\Xi$ is admissible will then follow by Borel-Cantelli's Lemma.
\begin{eqnarray*}
\Prob{\sum_{i=1}^{n^{1/4}}Z_i>1-\frac{1}{\sqrt{n}}}&\geq&\Prob{\sum_{i=1}^{n^{1/4}}\bar Z_i>1-\frac{1}{\sqrt{n}}}\\ 
&=&\Prob{\prod_{i=1}^{n^{1/4}}(1-Y_i)<\frac{1}{\sqrt{n}}}\\ 
&\geq & \Prob{C_{n^{1/4}}>n^{1/4}\frac{\delta}{2}}\mbb I_{\{(1-\epsilon)^{n^{1/4}\frac{\delta}{2}}<\frac{1}{\sqrt{n}}\}}.
\end{eqnarray*}
Important examples in this class are the Poisson-Dirichlet distribution, its two-parameter extension (see \cite{Fe10}).
\end{example}

The main result of this section is the following.
\begin{thm}\label{THM}
Let $(X_t:t\geq 0)$ be the solution to \eqref{dualcatsde} for given $(\pi_i)\in{\cal P}(\N\cup\{\infty\})$ such that $\beta=\sum_{k=1}^\infty k\pi_k<\infty$, for $0<\kappa<\infty$  and for a measure $\Xi$ such that $\widehat\Xi$ is an \emph{admissible} probability measure on $\nabla_{\infty}$. Let $(D_t)$ be the corresponding dual branching-coalescing process. With the same notation as in Proposition \ref{lemma1} and Lemma \ref{l:xibob}, let \begin{equation}
\kappa^*:=\frac{1}{2\beta}\Ex{\frac{1}{\sum_{i=1}^\infty {Z^*_i}^2}\frac{1}{W(1-W)}}.\end{equation}
So long as $\kappa^*<\infty$, $(D_t)$  has a unique stationary distribution if and only if $\kappa< \kappa^*.$ Otherwise, if $\kappa\geq \kappa^*$ then for every $n,m>0$ it holds that $\lim_{t\rightarrow \infty}\P(D_t<m\mid D_0=n)=0$. Equivalently, $\P_x(\lim_{t\rightarrow\infty} X_t=0)=1$ if and only if  $\kappa\geq \kappa^*$.
\end{thm}
To prove the claim we need the next Lemma which spells out the relationship between probability of extinction in the forward in time frequency process and stationarity of the dual ancestral process.
\begin{lemma}\label{lemma:selection21}
Let $\kappa>0$, $(\pi_i)$ and $\Xi$ satisfy the assumptions of Proposition \ref{lemma1} such that there exists a solution $(X_t)$ to \eqref{dualcatsde}. Let $(D_t)$ be the dual ancestral process to $(X_t)$. One of the following two cases is always true.
\begin{itemize}
\item[(i)] If $(D_t)$ is positive recurrent then $(D_t)$ has a unique stationary distribution $\mu$ and  
$$p(x):=\P_x\Big(\lim_{t\rightarrow\infty} X_t=0\Big)=1-\varphi_\mu(x),$$
where $\varphi_\mu$ is the probability generating function of $\mu$. 
In particular $p'(x)=-\sum_{m=1}^\infty \mu(m)mx^{m-1}$ is  strictly negative and decreasing for all $x\in[0,1]$.
\item[(ii)] If $D_t$ is not positive recurrent then $\P_x(\lim_{t\rightarrow\infty} X_t=0)=1$ for every $x\in(0,1]$. 
\end{itemize}
\end{lemma}
\begin{proof}
First assume that $(D_t)$ is positive recurrent. The process $(D_t)$ moves from any $n\in \N$ to each of its neighbours in $\N\cup\{\infty\}$ with a positive rate, hence clearly $(D_t)$ is an irreducible Markov process and thus it has a unique stationary distribution. Now we apply moment duality and dominated convergence. Let $n\in \N\cup\{\infty\}$ and $x\in(0,1]$.
\begin{equation}\label{eq:recurrent}
\E_x\Big[\lim_{t\rightarrow \infty}X_t^n\Big]=\lim_{t\rightarrow \infty}\E_x\Big[X_t^n\Big]=\lim_{t\rightarrow \infty}\E_n\Big[x^{D_t}\Big]=\sum_{m=1}^\infty \mu(m)x^m=\varphi_\mu(x).
\end{equation}
Since the random variable $\lim_{t\rightarrow \infty}X_t$ takes values in $[0,1]$, its distribution is characterised by its moments. This allows us to conclude that 
$$
\Prob{\lim_{t\rightarrow \infty}X_t\in\cdot |X_0=x}=\left(1-\varphi_\mu(x)\right)\delta_0(\cdot)+\varphi_\mu(x)\delta_1(\cdot).
$$
Now we assume that $(D_t)$ is not positive recurrent. This implies that for every $n,m\in \N\cup\{\infty\},$ $\lim_{t\rightarrow \infty} \P(D_t <m\mid D_0=n)=0$. We will use again moment duality and dominated convergence. For any $x\in(0,1],$
\begin{equation}\label{eq:nonrecurrent}
\E_x\Big[\lim_{t\rightarrow \infty}X_t^n\Big]=\lim_{t\rightarrow \infty}\E_x\Big[X_t^n\Big]=\lim_{t\rightarrow \infty}\E_n\Big[x^{D_t}\Big]\leq x^m.
\end{equation}
As $m$ is arbitrary, we conclude that $\E_x\Big[\lim_{t\rightarrow \infty}X_t^n\Big]=0$ for every $n\in \N$. This implies that
$$
\Prob{\lim_{t\rightarrow \infty}X_t\in\cdot |X_0=x}=\delta_0(\cdot).
$$
\end{proof}
We are now ready to prove Theorem \ref{THM}.
\begin{proof}
We will first assume that $\Xi$ is concentrated on the $m$-dimensional simplex i.e. that $\widehat \Xi$ selects almost surely at most $m$ non-zero atoms $Z_1,\ldots,Z_m$. \\

\textit{Sufficiency.} Take $\kappa=\kappa^*.$ Let $p(x)=\P(\lim_{t\rightarrow \infty}X_t=0 |X_0=x)$ where $X_t$ has generator $A$ given by \eqref{dualc}.

Note that $Ap(x)=0$. Then, by Lemma \ref{l:xibob}, 
\begin{eqnarray*}\label{dualcat1}
0&=&\kappa\sum_{k=1}^\infty {\pi_k}(x^{k+1}-x)p'(x)\\&&+\frac{1}{2}\Ex{\frac{\left(-x+\sum_{i=1}^mZ^*_iB_i\right)\left(\sum_{i=1}^m Z^*_iB_i\right)}{\sum_{i=1}^m {Z^*_i}^2}p''\left(x(1-W)+VW\sum_{i=1}^m {Z^*_i}B_i\right)}.
\end{eqnarray*}
Taking the expectation with respect to $V$, dividing by $\frac{1}{2}x(1-x)$ and observing that $\frac{x^{i+1}-x}{x(1-x)}=-\sum_{j=0}^{i-1}x^j$ we obtain
\begin{eqnarray}\label{pprime}
0&=&-2\kappa\sum_{k=1}^\infty {\pi_k}\sum_{j=0}^{k-1}x^jp'(x)\notag\\
&&+\frac{1}{x(1-x)}\Ex{\frac{\left(-x+\sum_{i=1}^m Z^*_iB_i\right)}{W\sum_{i=1}^m {Z^*_i}^2}\left\{p'\left(x(1-W)+W\sum_{i=1}^m Z^*_iB_i\right)-p'(x(1-W))\right\}}.\nonumber
\end{eqnarray}
Writing the expectation over $B=(B_1,B_2...)$ as a sum  we obtain
\begin{eqnarray*}
0&=&-2\kappa\sum_{k=1}^\infty {\pi_k}\sum_{j=0}^{k-1}x^jp'(x)\label{dualcat3}\\
&&+\sum_{b\in\{0,1\}^m}\prod_{i=1}^m \Prob{B_i=b_i}\Ex{\frac{\left(-x+\sum_{i=1}^m Z^*_ib_i\right)}{{x(1-x)}W\sum_{i=1}^m {Z^*_i}^2}\Delta(x,b;W,Z^*)},\nonumber
\end{eqnarray*}
where
$$\Delta(x,b;W,Z^*):=\left\{p'\left(x(1-W)+W\sum_{i=1}^mZ^*_ib_i\right)-p'(x(1-W))\right\}.$$
Since $\P(B=b)=x^{|b|}(1-x)^{m-|b|}$ where $|b|=\sum_{i=1}^m b_i$ then, if we consider the limit when $x$ goes to $1$, only the terms where at most one of the $b_i$ is equal to zero remain, as only in this cases $\lim_{x\rightarrow 1}\frac{\P(b=v)}{x(1-x)}(-x+\sum_{i=1}^mz_ib_i)\neq 0$. Thus as $x\to 1,$
\begin{eqnarray}\label{finitem}
0&=&-2\kappa\sum_{k=1}^\infty {\pi_k} \ kp'(1^-)\nonumber\\
&&+\Ex{\frac{p'(1^-)-p'(1-W)}{W\sum_{i=1}^m {Z^*_i}^2}}\nonumber\\
&&
-\sum_{i=1}^m\Ex{\frac{Z^*_i}{W\sum_{i=1}^m {Z^*_i}^2}\left\{p'(1-WZ^*_i)-p'(1-W)\right\}}\
\end{eqnarray}
The first expectation accounts for the case where $b_i=1$ for all $i=1,\ldots,m$. The second expectation deals with all the cases where the only zero coordinate in $b=(b_1,\ldots,b_m)$ is $b_i$ for each $i=1,\ldots m$, in each of which cases one has 
$$-x+\sum_{i=1}^m Z^*_i b_i=-x+1-Z^*_i\to-Z_i^*,\ \ \ x\to 1.$$
Rewriting,
\begin{eqnarray}
0&=&-2\kappa\beta p'(1^-)\nonumber
\\&&+\sum_{i=1}^m\Ex{\frac{Z^*_i}{W\sum_{i=1}^m {Z^*_i}^2}\left\{p'(1^-)-p'(1-Z^*_iW)\right\}}.\nonumber
\end{eqnarray}
{Multiplying and dividing by $(1-W)$ the argument of the expectation, our choice of $\kappa$ implies that
\begin{eqnarray}\label{dualcat4}
    0=\sum_{i=1}^m\Ex{\frac{-Wp'(1)-(1-W)p'(1-Z^*_iW)}{W(1-W)}}.
\end{eqnarray}}
Equation \eqref{dualcat4} cannot be true if $(D_t)$ is positive recurrent, because in that case $p'(x)$ is negative. By Lemma \ref{lemma:selection21}, then, $p(x)=1$ for all $x\in[0,1)$ and thus $p'(x)=0$ for all $x\in[0,1]$. \\

\textit{Necessity.} Let $\kappa<\kappa^*$. Assume that $p(x)=1$ for all $x\in [0,1)$. We will show that this assumption leads to a contradiction. Consider the test function $w(x)=\log(1-x)+Kx$, where $K$ is a positive (large) constant. For any $t\geq 0$ consider the generator equation
\begin{equation}\label{eq:selection81}
\E_x[w(X_t)]-w(x)=\int_0^t\E_x[Aw(X_s)]ds,
\end{equation}
where $A$ is the generator in Equation \eqref{dualc}. If $p(x)=1$, when $t\rightarrow \infty$, equation \eqref{eq:selection81} becomes
\begin{equation}\label{eq:selection91}
-w(x)=\int_0^\infty \E_x[Aw(X_s)]ds,
\end{equation}
we will show that for some choice of $x\in (0,1]$ and $K>0$ equation \eqref{eq:selection91} does not hold. This implies that $p(x)\neq 1$ for some $x\in(0,1]$, a contradiction. Lemma \ref{lemma:selection21} then assures that $p(x)\neq 1$ for all $x\in(0,1]$.

Our proof is based on the analysis of $\E[Aw(X_s)]$. We will first take the expectation with respect to the uniform random variable $V$. Using $w'(x)=-\frac{1}{1-x}+K$ and the fact that
$$\int_0^1 w^{''}(a+bv)\ dv=\frac{1}{b}[w'(a+b)-w'(a)]$$
we obtain
\begin{eqnarray}\label{eq:selection111}
&&Aw(x)=x\kappa\sum_{i=1}^\infty \pi_i\sum_{j=0}^{i-1}x^j-xK\kappa\sum_{i=1}^\infty \pi_i(1-x^{i})\nonumber\\&&
-\frac{1}{2}\Ex{\frac{\left(-x+\sum_{i=1}^m{Z^*_i}B_i\right)\left[(1-x(1-W)-W\sum_{i=1}^m{Z^*_i}B_i)^{-1}-(1-x(1-W))^{-1}\right]}{W\sum_{i=1}^m {Z^*_i}^2}}\nonumber\\
&=&x\kappa\sum_{i=1}^\infty \pi_i\sum_{j=0}^{i-1}x^j-xK\kappa\sum_{i=1}^\infty \pi_i(1-x^{i})\nonumber\\&&
-\frac{1}{2}\Ex{\frac{1}{\sum_{i=1}^m {Z^*_i}^2}(-x+\sum_{i=1}^m{Z^*_i}B_i)\frac{\sum_{i=1}^m{Z^*_i}B_i}{(1-x(1-W)-W\sum_{i=1}^m{Z^*_i}B_i)(1-x(1-W))}}.\nonumber\\
\end{eqnarray}
Note that for all $x\in (0,1)$ if $K\rightarrow \infty$ the right hand side of Equation \eqref{eq:selection111} tends to $-\infty$. Then we can chose $K$ large enough such that, for $x\in[ 0,1)$,
 $$x^{-1}A w(x)<\lim_{x\to 1} x^{-1} Aw(x).$$ We shall prove that  
 $$
\lim_{x\to 1}  Aw(x)<0
$$
 so that  $Aw(x)<0$ for every $x\in[0,1)$. \\
 Rewriting the expectation over $(B_1,\ldots,B_m)$ as a sum and reordering the terms, \eqref{eq:selection111} becomes
\begin{eqnarray}\label{eq:selection11}
&&Aw(x)=x\kappa\sum_{i=1}^\infty \pi_i\sum_{j=0}^{i-1}x^j-xK\kappa\sum_{i=1}^\infty \pi_i(1-x^{i})\notag\\&&
-\frac{1}{2}\sum_{b\in\{0,1\}^m}x^{|b|}(1-x)^{m-|b|}\notag\\
&&\hspace{2cm}\times\Ex{\frac{1}{\sum_{i=1}^m {Z^*_i}^2}\left(-x+\sum_{i=1}^mZ^*_ib_i\right)\frac{\sum_{i=1}^mZ^*_ib_i}{\left(1-x(1-W)-W\sum_{i=1}^mZ^*_ib_i\right)(1-x(1-W))}}.\nonumber\\
\end{eqnarray}
But with $K$ chosen large enough, as $\lim_{x\rightarrow 1}\Prob{\bigcap_{i=1}^m\{B_i=1\}}=1$,
\begin{eqnarray}\label{eq:selection12}
\lim_{x\rightarrow 1} Aw(x)&=&\kappa\sum_{k=1}^\infty k\pi_k\nonumber\\&&-\lim_{x\rightarrow 1}\frac{1}{2}\Ex{\frac{1}{\sum_{i=1}^m {Z^*_i}^2}(-x+1)\frac{1}{(1-x(1-W)-W)(1-x(1-W))}}\nonumber\\
&=&\kappa\beta-\frac{1}{2}\Ex{\frac{1}{\sum_{i=1}^m {Z^*_i}^2}\frac{1}{W(1-W)}}\\&=&
\kappa\beta-\kappa^*\beta<0.\nonumber
\end{eqnarray}
This implies that, for every $x\in [0,1],$ $Aw(x)<0$. Finally we observe that  $\lim_{x\rightarrow 1}w(x)=-\infty$, so for every big enough $x\in [0,1)$, it holds that $-w(x)>0$. However, we showed that $\int_0^\infty \E_x[Aw(X_s)]ds<0$ for every $x\in(0,1)$. This contradicts Equation \eqref{eq:selection91} and completes the proof for $\widehat\Xi$ a probability measure selecting at most $m$ non-zero atoms with probability one.

Our new task is to extend the proof to the case which $\widehat\Xi$ is a general admissible probability measure, not necessarily concentrated on the $m$-simplex. \\

\textit{Sufficiency.} The proof remains unchanged until Equation \eqref{dualcat3}. There we will fix $n\in\N$ and write the expectation over $(B_i)$ by summing over the probability of every posible configuration of the first $m_n$ terms.  
\begin{eqnarray}
0&=&-2\beta\sum_{i=1}^\infty \pi_i\sum_{j=0}^{i-1}x^jp'(x)\\
&&+\sum_{b\in\{0,1\}^{m_n}}\Prob{\bigcap_{i=1}^{m_n}\{B_i=b_i\}}
\E\Big[\frac{1}{x(1-x)}\frac{\left(-x+\sum_{i=1}^{m_n}{Z^*_i}b_i+\sum_{i=m_n+1}^{\infty}{Z^*_i}B_i\right)}{\sum_{i=1}^\infty {Z^*_i}^2}\nonumber\\&&
\times\frac{p'\left(x(1-W)+W(\sum_{i=1}^{m_n}{Z^*_i}b_i+\sum_{i=m_n+1}^{\infty}{Z^*_i}B_i)\right)-p'(x(1-W))}{W}\Big].\nonumber
\end{eqnarray}

Rather than taking the limit when $x$ goes to $1$, we evaluate in $x=1-1/n$. Note that $\sum_{i=1}^{m_n}{Z^*_i}+\sum_{i=m_n+1}^{\infty}{Z^*_i}B_i=1-Ec_n$, where $c_n$ is the sequence making admissible the probability $\widehat\Xi$ (see Definition \ref{def:admiss}) and E is a random variable such that $|E|\leq 2$. \\
Now, contrary to the $m$-finite case, we will need to control all the terms in the expectation for every possible realisation $b=b_1,\ldots,b_{m_n}$ and not only for those binary vectors $b$ with at most one zero coordinate.\\
Let us first see what happens in the case $b=(1,1,...,1)$. 
\begin{eqnarray}
&&\E\Big[\frac{(1-1/n)^{m_n-1}}{1/n}\frac{(1/n-E c_n)}{\sum_{i=1}^\infty {Z^*_i}^2}\nonumber\\ &&\times\frac{p'((1-1/n)(1-W)+W(1-E c_n)-p'((1-1/n)(1-W))}{W}\Big].\nonumber
\end{eqnarray}
Since, by definition \ref{def:admiss}, $nc_n\to 0,$ if we let $n$ go to infinity this term converges to
\begin{eqnarray}\label{all1}
\Ex{\frac{1}{\sum_{i=1}^\infty {Z^*_i}^2}\frac{p'((1-W)+W)-p'(1-W)}{W}}.
\end{eqnarray}
Now we will focus in all the vectors $b\in\{0,1\}^{m_n}$ that have exactly one zero. 
\begin{eqnarray}
\sum_{j=1}^{m_n}\E\Big[(1-1/n)^{m_n-2}\frac{1}{\sum_{i=1}^\infty {Z^*_i}^2}(1/n-{Z^*_i}-Ec_n)\nonumber
\\ \times\frac{p'((1-1/n)(1-W)+W(1-{Z^*_j}-Ec_n)-p'((1-1/n)(1-W))}{W}\Big].\nonumber
\end{eqnarray}
When $n\to\infty$ the equation reduces to
\begin{eqnarray}\label{one0}
\sum_{j=1}^{m_n}\Ex{\frac{-Z^*_j}{\sum_{i=1}^\infty {Z^*_i}^2}\frac{p'((1-W{Z^*_j}))-p'(1-W)}{W}}.
\end{eqnarray}
Finally, let $R_k:=\{b\in\{0,1\}^{m_n}:|b|=m_n-k\}$, for every $k\in\{0,1,...,m_n\}$. Note that $|R_k|\leq m_n^k=o(n^{k/2})$ and for every $k\geq 2$
\begin{eqnarray}\label{neg}
\sum_{b\in R_k}\E\Big[(1-1/n)^{m_n-k-1}(1/n)^{k-1}\frac{\left(-1+1/n+\sum_{i=1}^{m_n}Z^*_ib_i+\sum_{i=m_n+1}^\infty Z^*_iB_i\right)}{\sum_{i=1}^\infty {Z^*_i}^2}\nonumber
\\\times\frac{p'\left((1-1/n)(1-W)+W(\sum_{i=1}^{m_n}Z^*_ib_i+\sum_{i=m_n+1}^\infty Z^*_iB_i)\right)-p'((1-1/n)(1-W))}{W}\Big]\nonumber\\
=O\big(m_n^k(1/n)^{k-1}\big)=o\big(n^{k/2-1}\big).\nonumber\\
\end{eqnarray}
In the derivation of the equality, we used the property $-p'(1)<\infty$. This is a fact which requires a proof. Assume $p'(1)=-\sum_{k=1}^\infty k\mu(k)=-\infty$. Consider the limit when $x\rightarrow 1$ in \eqref{pprime}. Using that $-x+\sum_{i=1}^m{Z^*_i}b_i<1-x$ and that $p'(x)$ is negative and decreasing, we obtain
\begin{eqnarray}\label{infinitzero}
0&>&\frac{1}{2}\Ex{\frac{-p'(1)-(1-W)p'(x(1-W)+W\sum_{i=1}^m{Z^*_i}b_i)-(1-W)p'(x(1-W))}{W\sum_{i=1}^\infty {Z^*_i}^2}}\nonumber\\
&=&\infty,\nonumber
\end{eqnarray}
which is a contradiction. Then $-p'(1)<\infty$ and Equation \eqref{neg} holds.

Equations \eqref{all1}, \eqref{one0} and \eqref{neg} imply that as $n\to\infty$ we obtain Equation \eqref{finitem}, and the proof follows as in the finite $m$ case.\\
\textit{Necessity.} The strategy is the same as in the \textit{sufficiency} case. Observe that 
\begin{eqnarray*}
&&\left(1-x(1-W)-W\left(\sum_{i=1}^{m_n}{Z^*_i}B_i+\sum_{i=m_n+1}^{\infty}z_iB_i\right)\right)\notag\\
&&\ \ \ \ \ =(1-W)\left(-x+\sum_{i=1}^{m_n}{Z^*_i}B_i+\sum_{i=m_n+1}^\infty Z^*_iB_i\right)+1-\left(\sum_{i=1}^{m_n}{Z^*_i}B_i+\sum_{i=m_n+1}^\infty Z^*_iB_i\right)\notag\\
&&\ \ \ \ \ =(1-W)\left(-x+\sum_{i=1}^{m_n}Z^*_iB_i+\sum_{i=m_n+1}^{\infty}Z^*_iB_i\right)+O(c_n).
\end{eqnarray*}
Applying this to Equation \eqref{eq:selection11}, and evaluating at $x=1-1/n$, we obtain
\begin{eqnarray}
Av(x)&=&\kappa\sum_{k=1}^\infty k\pi_k\notag\\
&& -\frac{1}{2}\Prob{\bigcap_{i=1}^{m_n}\{B_i=1\}}
\Ex{\frac{1}{\sum_{i=1}^\infty {Z^*_i}^2}\frac{1}{(1-W)(1-(1-1/n)(1-W))}}+O(c_n)\nonumber\\&&
=\kappa\beta-\frac{1}{2}\Ex{\frac{1}{W(1-W)\sum_{i=1}^\infty {Z^*_i}^2}}+O(c_n),\nonumber
\end{eqnarray}
which converges to \eqref{eq:selection12} when $n\to\infty$. The rest of the proof is as in the finite $m$ case.
\end{proof}

\section*{Acknowledgements} We thank Bob Griffiths and Cl\'ement Foucart for many useful conversations. The second author is also grateful to Jochen Blath for his friendly hospitality at the TU Berlin, Institut f\"ur Mathematik, where part of this research has been carried out.

\bibliographystyle{abbrv}

\end{document}